\documentclass{amsart}

\usepackage{amsmath,amsfonts,amssymb,stmaryrd,shuffle}
\usepackage{tikz}
\usetikzlibrary{decorations.markings}
\usepackage{xcolor}

\usepackage{fullpage}

\newtheorem{theorem}{Theorem}[section]
\newtheorem{lemma}[theorem]{Lemma}
\newtheorem{proposition}[theorem]{Proposition}
\newtheorem{corollary}[theorem]{Corollary}

\theoremstyle{definition}
\newtheorem{definition}[theorem]{Definition}

\theoremstyle{remark}
\newtheorem{remark}[theorem]{Remark}

\numberwithin{equation}{section}

\newcommand{\newword}[1]{\textbf{\emph{#1}}}

\newcommand{\st}[1]{{#1}^{\prime}}

\newcommand{\row}{\mathrm{row}}
\newcommand{\hook}{\mathrm{hook}}

\newcommand{\Des}{\mathrm{Des}}

\newcommand{\SYT}{\mathrm{SYT}}
\newcommand{\SST}{\mathrm{SST}}

\newcommand{\ipo}{i\!+\!1}
\newcommand{\imo}{i\!-\!1}

\newlength\cellsize 

\savebox2{%
\begin{picture}(10,10)
\put(0,0){\line(1,0){10}}
\put(0,0){\line(0,1){10}}
\put(10,0){\line(0,1){10}}
\put(0,10){\line(1,0){10}}
\end{picture}}

\newcommand\cellify[1]{\def\thearg{#1}\def\nothing{}%
\ifx\thearg\nothing\vrule width0pt height\cellsize depth0pt%
  \else\hbox to 0pt{\usebox2\hss}\fi%
  \vbox to 10\unitlength{\vss\hbox to 10\unitlength{\hss$#1$\hss}\vss}}

\newcommand\tableau[1]{\vtop{\let\\=\cr
\setlength\baselineskip{-10000pt}
\setlength\lineskiplimit{10000pt}
\setlength\lineskip{0pt}
\halign{&\cellify{##}\cr#1\crcr}}}

\definecolor{boxgray}{gray}{.85}
\newcommand{\cb}{\color{boxgray}\rule{1\cellsize}{1\cellsize}\hspace{-\cellsize}\usebox2}

\savebox4{%
\begin{picture}(10,10)
\put(0,0){\color{boxgray}\rule{1\cellsize}{1\cellsize}}
\put(0,0){\line(1,0){10}}
\put(0,0){\line(0,1){10}}
\put(10,0){\line(0,1){10}}
\put(0,10){\line(1,0){10}}
\end{picture}}

\newcommand\grayify[1]{\def\thearg{#1}\def\nothing{}%
\ifx\thearg\nothing\vrule width0pt height\cellsize depth0pt%
  \else\hbox to 0pt{\usebox4\hss}\fi%
  \vbox to 10\unitlength{\vss\hbox to 10\unitlength{\hss$#1$\hss}\vss}}

\newlength\widecellsize 

\savebox3{%
\begin{picture}(16,16)
  \linethickness{0.175mm}
  \put(0,0){\line(1,0){16}}
  \put(0,0){\line(0,1){16}}
  \put(16,0){\line(0,1){16}}
  \put(0,16){\line(1,0){16}}
\end{picture}}

\newcommand\widecellify[1]{\def\thearg{#1}\def\nothing{}%
\ifx\thearg\nothing\vrule width0pt height\widecellsize depth0pt%
  \else\hbox to 0pt{\usebox3\hss}\fi%
  \vbox to \widecellsize{\vss\hbox to \widecellsize{\hss$_{#1}$\hss}\vss}}

\newcommand\widetableau[1]{\vtop{\let\\=\cr
\setlength\baselineskip{-16000pt}
\setlength\lineskiplimit{16000pt}
\setlength\lineskip{0pt}
\halign{&\widecellify{##}\cr#1\crcr}}}

\usepackage[colorinlistoftodos]{todonotes}

\begin{document}


\title{Queer dual equivalence graphs}  

\author{Sami H. Assaf}
\address{Department of Mathematics, University of Southern California, 3620 S. Vermont Ave., Los Angeles, CA 90089-2532, U.S.A.}
\email{shassaf@usc.edu}
\thanks{Work supported in part by NSF grant DMS-1763336.}





\begin{abstract}
  We introduce a new paradigm for proving the Schur $P$-positivity. Generalizing dual equivalence, we give an axiomatic definition for a family of involutions on a set of objects to be a queer dual equivalence, and we prove whenever such a family exists, the fundamental quasisymmetric generating function is Schur $P$-positive. In contrast with shifted dual equivalence, the queer dual equivalence involutions restrict to a dual equivalence when the queer involution is omitted. We highlight the difference between these two generalization with a new application to the product of Schur $P$-functions.
\end{abstract}

\maketitle
\tableofcontents

%
\section{Introduction}
%
\label{sec:introduction}

Schur $P$-polynomials arise as characters of tensor representations of the queer Lie superalgebra \cite{Ser84}, characters of projective representations of the symmetric group \cite{Ste89}, and representatives for cohomology classes dual to Schubert cycles in isotropic Grassmannians \cite{Pra91}. These polynomials enjoy many nice properties parallel to Schur polynomials; in particular, Schur $P$- and $Q$-polynomials form dual bases for an important subspace of symmetric functions, are Schur positive \cite{Sag87}, and the former have positive structure constants \cite{Ste89}.

Many polynomials that arise in representation theoretic or geometric contexts can be expressed as a non-negative sum of Schur $P$-polynomials including type B Stanley symmetric functions, a generalization of Stanley symmetric functions \cite{Sta84} introduced by Billey and Haiman \cite{BH95}, and involution Stanley symmetric functions introduced by Hamaker, Marberg, and Pawlowski \cite{HMP17}. Thus, as with the quintessential problem of Schur positivity, establishing Schur $P$-positivity of a given function is a recurring problem with broad application.

Assaf \cite{Ass18} and Billey, Hamaker, Roberts, and Young \cite{BHRY14} simultaneously defined a shifted generalization of dual equivalence \cite{Ass15}, a machinery developed by Assaf to prove the Schur positivity of a function expressed in terms of fundamental quasisymmetric functions. Briefly, a \emph{shifted dual equivalence} is a family of involutions on a set of objects endowed with a \emph{peak} set satisfying certain local conditions that ensure the generating polynomial of the objects is Schur $P$-positive. Given the close relationship between Schur $P$- and $Q$-functions, this also gives a means to prove Schur $Q$-positivity.

In \cite{Ass18}, the author uses dual equivalence to establish the Schur positivity of Schur $P$-functions and shifted dual equivalence to establish the Schur $Q$-positivity of a skew Schur $Q$-function. Using duality with Schur $P$-functions, this establishes the positivity of structure constants for Schur $P$-functions. In this paper, we define a new generalization of dual equivalence which we term \emph{queer dual equivalence}, which we argue is better suited to Schur $P$-functions, whereas shifted dual equivalence as defined in \cite{Ass18,BHRY14} is better suited to Schur $Q$-functions.

Returning to the original notion of dual equivalence on objects endowed with \emph{descent} sets, instead of altering these indexing functions, we enrich the existing dual equivalence involutions with one additional involution, the \emph{queer involution}, so that the accompanying conditions on the queer involution ensure that the generating polynomial of the objects is Schur $P$-positivity. Thus we have a new paradigm for establishing Schur $P$-positivity. We apply this machinery directly to products of Schur $P$-functions to give a new proof of the positivity of structure constants for Schur $P$-functions.

One of the main features of dual equivalence and its shifted variation is the existence of \emph{local} conditions on the involutions. The main impediment to wider applications of queer dual equivalence is the needed conditions are global. However, taking inspiration from the parallel story of queer crystals \cite{GJKK10,GJKKK10} for the quantized universal enveloping algebras for the queer superalgebra \cite{Ser84}, one can expect that an alternative characterization with only local conditions should exist. Indeed, using the explicit connection between crystals and dual equivalences in type A \cite{Ass08}, such a local characterization might also yield a resolution to local characterization sought by Assaf and Oguz \cite{AO20} for the queer crystal.

%
\section{Symmetric Functions}
%
\label{sec:deg}

A \newword{partition} $\lambda = (\lambda_1,\lambda_2, \ldots, \lambda_{\ell})$ is a weakly decreasing sequence of positive integers, i.e. $\lambda_1 \geq \lambda_2 \geq \cdots \geq \lambda_{\ell} > 0$. The \newword{size} of a partition is the sum of its parts, i.e. $\lambda_1 + \lambda_2 + \cdots + \lambda_{\ell}$. We identify a partition $\lambda$ with its Young diagram, the collection of left-justified cells with $\lambda_i$ cells in row $i$, indexed from the bottom (French notation). 

A \newword{permutation} of $[n]$ is an ordering of the numbers $[n] = \{1,2,\ldots,n\}$. A \newword{standard Young tableau} of shape $\lambda$ is a bijective filling of the Young diagram for $\lambda$ with letters of a permutation such that entries weakly increase along rows and strictly increase up columns. The \newword{row reading word} of a standard Young tableau $T$, denoted by $\row(T)$, is the permutation obtained by reading the rows of $T$ left to right, from top to bottom. For example, see Fig.~\ref{fig:SYT}.

\begin{figure}[ht]
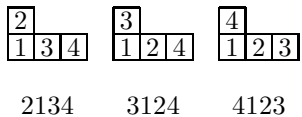

  \begin{center}
    \begin{displaymath}
      \begin{array}{ccc}
        \tableau{2 \\ 1 & 3 & 4} &
        \tableau{3 \\ 1 & 2 & 4} &
        \tableau{4 \\ 1 & 2 & 3} \\ \\
        2134 &
        3124 &
        4123
      \end{array}
    \end{displaymath}
    \caption{\label{fig:SYT}The standard Young tableaux of shape $(3,1)$ with (row) reading words indicated below.} 
  \end{center}
\end{figure}

The \newword{descent set} of a permutation $w$, denoted by $\Des(w)$, is given by
\begin{equation}
  \Des(w) \ = \ \left\{ i \ | \ \mbox{$i$ right of $i+1$} \right\}.
  \label{e:des}
\end{equation}
When $w$ is a permutation of length $n$, we have $\Des(w) \subseteq [n-1]$. When we wish to emphasize $n$, we write $\Des_n$. The \newword{descent set} of a standard Young tableau is the descent set of its row reading word. For example, the descent sets of the tableaux in Fig.~\ref{fig:SYT} from left to right are $\{1\}, \{2\}, \{3\}$. 

\begin{definition}[\cite{Ges84}]
  For $D$ a subset of $[n-1]$, the \newword{fundamental quasisymmetric function} $F_D$ is
  \begin{equation}
    F_{D}(X) \ = \ \sum_{\substack{i_1 \leq \cdots \leq i_n \\ j \in D \Rightarrow i_j < i_{j+1}}} x_{i_1} \cdots x_{i_n} .
    \label{e:quasisym}
  \end{equation}
\end{definition}

Implicit in our notation is that $D \subseteq [n-1]$. When we wish to make this explicit, we write $F_{n,D}$ or $F_{D_n}$. 

Gessel \cite{Ges84} defined the fundamental basis for quasisymmetric functions precisely to capture the following expansion for Schur functions arising from Stanley's Fundamental Theorem for $P$-partitions \cite{Sta72}.

\begin{definition}
  For $\lambda$ a partition of $n$, the \newword{Schur function} $s_{\lambda}$ is 
  \begin{equation}
    s_{\lambda}(X) \ = \ \sum_{T \in \SYT(\lambda)} F_{\Des(\row(T))}(X),
    \label{e:schur}
  \end{equation}
  where $\SYT(\lambda)$ denotes the set of all standard Young tableaux of shape $\lambda$.
\label{def:schur}
\end{definition}

For example, from Fig.~\ref{fig:SYT}, we may compute
\begin{eqnarray*}
s_{(3,1)} & = & F_{\{1\}} + F_{\{2\}} + F_{\{3\}} .
\end{eqnarray*}

A \newword{strict partition} $\gamma$ is one whose parts strictly decrease, i.e. $\gamma_1 > \gamma_2 > \cdots > \gamma_{\ell} > 0$. We identify a strict partition $\gamma$ with its shifted diagram obtained by shifting row $i$ of its Young diagram $i-1$ columns to the right.

A \newword{signed permutation} of $[n]$ is permutation of $[n]$ where each letter is allowed to be marked or unmarked. A \newword{signed standard tableau} of strict shape $\gamma$ is a bijective filling of the shifted Young diagram for $\gamma$ with letters of a signed permutation such that entries weakly increase along rows and columns with no signed entries on the main diagonal. The \newword{hook reading word} of a signed standard tableau $S$, denoted by $\hook(S)$, is the permutation obtained by reading the marked entries of $i$th column of $S$ from bottom to top and then the $i$th row of $S$ from left to right, for $i$ from largest to smallest. For examples, see Fig.~\ref{fig:SST}. 

\begin{figure}[ht]
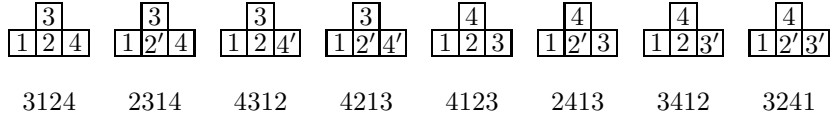

  \begin{center}
    \begin{displaymath}
      \begin{array}{cccccccc}
      \tableau{& 3 \\ 1 & 2 & 4} &
      \tableau{& 3 \\ 1 & \st{2} & 4} &
      \tableau{& 3 \\ 1 & 2 & \st{4}} &
      \tableau{& 3 \\ 1 & \st{2} & \st{4}} &
      \tableau{& 4 \\ 1 & 2 & 3} &
      \tableau{& 4 \\ 1 & \st{2} & 3} &
      \tableau{& 4 \\ 1 & 2 & \st{3}} &
      \tableau{& 4 \\ 1 & \st{2} & \st{3}} \\ \\
      3124 &
      2314 &
      4312 &
      4213 &
      4123 &
      2413 &
      3412 &
      3241
      \end{array}
    \end{displaymath}
    \caption{\label{fig:SST}The signed standard tableaux of shape $(3,1)$ with (hook) reading words indicated below.} 
  \end{center}
\end{figure}

The \newword{descent set} of a signed standard tableau is the descent set of its hook reading word. For example, the descent sets of the tableaux in Fig.~\ref{fig:SST} from left to right are $\{2\}, \{1\}, \{2,3\}, \{1,3\}, \{3\}, \{1,3\}, \{2\}, \{1,2\}$. 

We take the following result of Assaf \cite{Ass18}(Proposition 4.2) as our definition for Schur $P$-functions, parallel to Definition~\ref{def:schur} for Schur functions.

\begin{definition}
  For $\gamma$ a strict partition of $n$, the \newword{Schur $P$-function} $P_{\gamma}$ is 
  \begin{equation}
    P_{\gamma}(X) \ = \ \sum_{S \in \SST(\gamma)} F_{\Des(\hook(S))}(X),
    \label{e:schurP}
  \end{equation}
  where $\SST(\lambda)$ denotes the set of all signed standard tableaux of shape $\gamma$.
\label{def:schurP}
\end{definition}

For example, from Fig.~\ref{fig:SST}, we may compute
\begin{eqnarray*}
P_{(3,1)} & = & F_{\{1\}} + 2 F_{\{2\}} + F_{\{3\}} + F_{\{1,2\}} + 2 F_{\{1,3\}} + F_{\{2,3\}} .
\end{eqnarray*}

Since the Schur $P$-functions are symmetric, they can be expanded in the Schur basis. Stanley conjectured the expansion to be nonnegative, and this follows from Sagan's shifted insertion \cite{Sag87} independently developed by Worley \cite{Wor84}. 

\begin{theorem}[\cite{Sag87,Wor84}]
  For $\gamma$ a strict partition, we have
    \begin{equation}
      P_{\gamma} (X) \ = \ \sum_{\lambda} g_{\gamma,\lambda} s_{\lambda}(X)
      \label{e:P_schur}
    \end{equation}
    where $g_{\gamma,\lambda}$ are nonnegative integers.
\label{thm:P_pos}
\end{theorem}

For example, we may also express $P_{(3,1)}$ as 
\begin{eqnarray*}
P_{(3,1)} & = & s_{(3,1)} + s_{(2,2)} + s_{(2,1,1)} .
\end{eqnarray*}

%
\section{Dual equivalence}
%

Edelman and Greene \cite{EG87}(Definition~6.14) defined involutions on permutations that exchange $i$ and $i+1$ whenever $i-1$ or $i+2$ lies between them. Haiman \cite{Hai92} extended this by defining \newword{elementary dual equivalence involutions} $d_i$ on standard Young tableaux that swap $i$ with $i \pm 1$ whenever $i \mp 1$ lies in between them in the row reading word. For examples of these involutions on Young tableaux, see Fig.~\ref{fig:SYT-deg}.

\begin{figure}[ht]
  \begin{center}
    \begin{tikzpicture}[scale=1.6,
        label/.style={%
          postaction={ decorate,
            decoration={ markings, mark=at position 0.5 with \node #1;}}}]
      \node at (0,0) (A1) {$\tableau{2 \\ 1 & 3 & 4}$};
      \node at (1,0) (A2) {$\tableau{3 \\ 1 & 2 & 4}$};
      \node at (2,0) (A3) {$\tableau{4 \\ 1 & 2 & 3}$};
      \node at (3,0) (B1) {$\tableau{2 & 4 \\ 1 & 3}$};
      \node at (4,0) (B2) {$\tableau{3 & 4 \\ 1 & 2}$};
      \node at (5,0) (C1) {$\tableau{3 \\ 2 \\ 1 & 4}$};
      \node at (6,0) (C2) {$\tableau{4 \\ 2 \\ 1 & 3}$};
      \node at (7,0) (C3) {$\tableau{4 \\ 3 \\ 1 & 2}$};
      \draw[thick,color=red   ,label={[above]{$d_{2}$}}](A1) -- (A2) ;
      \draw[thick,color=blue,label={[above]{$d_{3}$}}]  (A2) -- (A3) ;
      \draw[thick,color=red   ,label={[above]{$d_{2}$}}](B1.05) -- (B2.175) ;
      \draw[thick,color=blue,label={[below]{$d_{3}$}}]  (B1.355) -- (B2.185) ;
      \draw[thick,color=blue,label={[above]{$d_{3}$}}]  (C1) -- (C2) ;
      \draw[thick,color=red   ,label={[above]{$d_{2}$}}](C2) -- (C3) ;
    \end{tikzpicture}
    \caption{\label{fig:SYT-deg}Dual equivalence for $\SYT(3,1) \cup \SYT(2,2) \cup \SYT(2,1,1)$.}
  \end{center}
\end{figure}
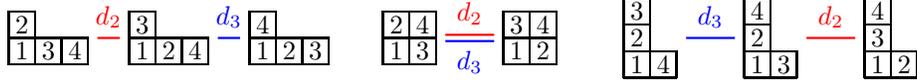

Two standard Young tableaux are \newword{dual equivalent} if one can be transformed into the other by a sequence of dual equivalence involutions. Haiman \cite{Hai92} proved that two standard Young tableaux are dual equivalent if and only if they have the same shape. Given this, we may rewrite Eq.~\eqref{e:schur} in terms of dual equivalence classes,
\begin{equation}
  s_{\lambda}(X) \ = \ \sum_{T \in [T_{\lambda}]} F_{\Des(T)}(X),
\label{e:classes}
\end{equation}
where $[T_{\lambda}]$ denotes the dual equivalence class of some fixed $T_{\lambda} \in \SYT(\lambda)$. 

This paradigm shift to summing over objects in a dual equivalence class is the basis for the universal method for proving that a quasisymmetric generating function is symmetric and Schur positive \cite{Ass07,Ass15}. 

Given a set of objects $\mathcal{A}$ and a notion of descents, one can form the quasisymmetric generating function for $\mathcal{A}$ by
\begin{displaymath}
  \sum_{T \in \mathcal{A}} F_{\Des(T)}(X).
\end{displaymath}
Two examples of this are Schur functions generated by standard Young tableaux Eq.~\eqref{e:schur} and Schur $P$-functions generated by signed standard tableau Eq.~\eqref{e:schurP}.

\begin{definition}[\cite{Ass15}]
  Let $\mathcal{A}$ be a finite set, and $\Des$ be a map from $\mathcal{A}$ to subsets of $[n-1]$. A \newword{dual equivalence for $(\mathcal{A},\Des)$} is a family of involutions $\{\varphi_i\}_{1<i<n}$ on $\mathcal{A}$ such that
  \renewcommand{\theenumi}{\roman{enumi}}
  \begin{enumerate}
  \item For all $0 \leq i-h \leq 3$ and all $T \in \mathcal{A}$, there exists a partition $\lambda$ of $i-h+3$ such that
    \[ \sum_{U \in [T]_{(h,i)}} F_{\Des_{(h-1,i+1)}(U)}(X) = s_{\lambda}(X), \]
    where $[T]_{(h,i)}$ is the equivalence class generated by $\varphi_h\ldots\varphi_i$ and $\Des_{(h,i)}(T)$ is the subset of $[i-h]$ obtained by deleting letters less than $h$ and weakly greater than $i$ from $\Des(T)$, and then subtracting $h-1$ from the letters that remain.
    
  \item For all $|i-j| \geq 3$ and all $T \in\mathcal{A}$, we have
    \begin{displaymath}
      \varphi_{j} \varphi_{i}(T) = \varphi_{i} \varphi_{j}(T).
    \end{displaymath}

  \end{enumerate}

  \label{def:deg}
\end{definition}

Assaf \cite{Ass15} showed that the involutions $d_i$ on standard Young tableaux satisfy Definition~\ref{def:deg} and that any involutions satisfying Definition~\ref{def:deg} have $\Des$-isomorphic equivalence classes. 

\begin{theorem}[\cite{Ass15}]
  If $\{\varphi_i\}$ is a dual equivalence for $(\mathcal{A},\Des)$, then
  \begin{equation}
    \sum_{T \in [U]} F_{\Des(T)}(X) \ = \ s_{\lambda}(X)
  \end{equation}
  for any $U \in \mathcal{A}$ and for some partition $\lambda$. In particular, the fundamental quasisymmetric generating function for $\mathcal{A}$ is symmetric and Schur positive.
  \label{thm:deg}
\end{theorem}

To demonstrate the utility of Theorem~\ref{thm:deg}, we recall the dual equivalence involutions on signed standard tableaux defined in \cite{Ass18}(Definition~6.1).

\begin{definition}[\cite{Ass18}]
  Let $S$ be a signed standard tableau. For $1<i<n$, let $a \leq b \leq c$ be the diagonals (row minus column if unmarked; column minus row if marked) on which $i-1,i,i+1$ reside. Then $d_i(w)$ is given by the following rule:
  \begin{enumerate}
  \item if $i$ lies on diagonal $b$, then $d_i(S) = S$;
  \item else if $a=b$ (respectively $b=c$), then toggle the sign on the letter on diagonal $c$ (respectively $a$);
  \item else if $\left| |a|-|c| \right| = 1$, then toggle the signs on the letters on diagonals $a$, $c$;
  \item else swap the absolute values of letters of diagonals $a$ and $c$, maintaining the signs in the given cells.
  \end{enumerate}
  \label{def:deg_shifted}
\end{definition}

The nontrivial cases of Definition~\ref{def:deg_shifted} are illustrated in Fig.~\ref{fig:sh-deg}. It is shown in \cite{Ass18} that case (2) arises only when $i-1$ and $i+1$ lie on the main diagonal. For case (3), the cells on diagonals $a$ and $c$ must form a horizontal or vertical domino. For case (4), the decorations indicate that the letters may be signed or unsigned, but the signs remain constant within the given cells. 

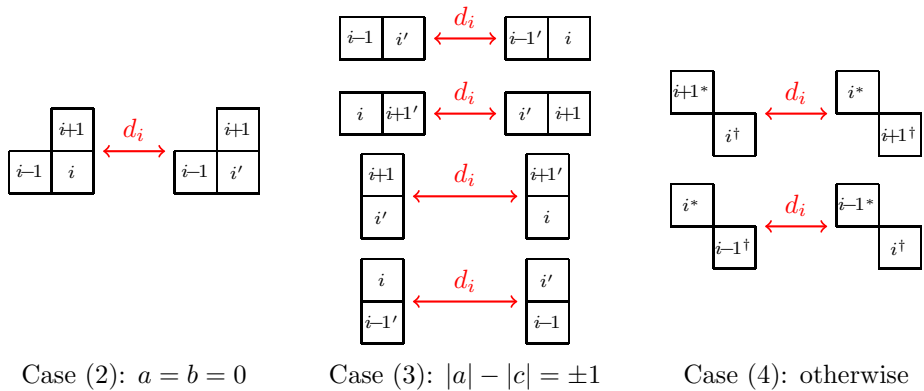
\begin{figure}[ht]
  \begin{center}
    \begin{tikzpicture}[xscale=2.2,yscale=1.0,
        label/.style={%
          postaction={ decorate,
            decoration={ markings, mark=at position 0.5 with \node #1;}}}]
      \node at (0,2)  (A2)  {$\widetableau{& \ipo \\ \imo & i }$};
      \node at (1,2)  (B2)  {$\widetableau{& \ipo \\ \imo & \st{i} }$};      
      \node at (2,3.5)(A3a) {$\widetableau{\imo & \st{i}}$};
      \node at (3,3.5)(B3a) {$\widetableau{\st{\imo} & i}$};      
      \node at (2,2.5)(A3b) {$\widetableau{i & \st{\ipo}}$};
      \node at (3,2.5)(B3b) {$\widetableau{\st{i} & \ipo}$};      
      \node at (2,1.4)(A3c) {$\widetableau{\ipo \\ \st{i}}$};
      \node at (3,1.4)(B3c) {$\widetableau{\st{\ipo} \\ i}$};      
      \node at (2,0)  (A3d) {$\widetableau{i \\ \st{\imo}}$};
      \node at (3,0)  (B3d) {$\widetableau{\st{i} \\ \imo}$};      
      \node at (4,1)  (A4a) {$\widetableau{i^{*} \\ & \imo^{\dagger}}$};
      \node at (5,1)  (B4a) {$\widetableau{\imo^{*} \\ & i^{\dagger}}$};      
      \node at (4,2.5)  (A4b) {$\widetableau{\ipo^{*} \\ & i^{\dagger}}$};
      \node at (5,2.5)  (B4b) {$\widetableau{i^{*} \\ & \ipo^{\dagger}}$};
      \node at (0.5,-1) {Case (2): $a=b=0$};
      \node at (2.5,-1) {Case (3): $|a|-|c|=\pm 1$};
      \node at (4.5,-1) {Case (4): otherwise};
      \draw[thick,color=red   ,<->,label={[above]{$d_{i}$}}](A2) -- (B2) ;
      \draw[thick,color=red   ,<->,label={[above]{$d_{i}$}}](A3a) -- (B3a) ;
      \draw[thick,color=red   ,<->,label={[above]{$d_{i}$}}](A3b) -- (B3b) ;
      \draw[thick,color=red   ,<->,label={[above]{$d_{i}$}}](A3c) -- (B3c) ;
      \draw[thick,color=red   ,<->,label={[above]{$d_{i}$}}](A3d) -- (B3d) ;
      \draw[thick,color=red   ,<->,label={[above]{$d_{i}$}}](A4a) -- (B4a) ;
      \draw[thick,color=red   ,<->,label={[above]{$d_{i}$}}](A4b) -- (B4b) ; 
    \end{tikzpicture}
    \caption{\label{fig:sh-deg}Illustration of the nontrivial cases for dual equivalence involutions on signed standard tableaux.}
  \end{center}
\end{figure}

For example, Fig.~\ref{fig:SST-deg} shows dual equivalence on signed standard tableaux of shape $(3,1)$. From left to right, the nontrivial actions are by cases (2), (4), (3), (4), and (2). Compare this structure with the dual equivalence structure on standard Young tableaux in Fig.~\ref{fig:SYT-deg}.

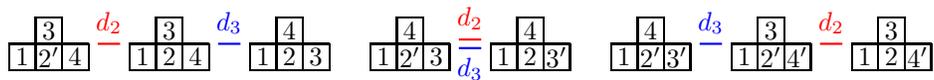
\begin{figure}[ht]
  \begin{center}
    \begin{tikzpicture}[scale=1.6,
        label/.style={%
          postaction={ decorate,
            decoration={ markings, mark=at position 0.5 with \node #1;}}}]
      \node at (0,0) (A1) {$\tableau{& 3 \\ 1 & \st{2} & 4}$};
      \node at (1,0) (A2) {$\tableau{& 3 \\ 1 & 2 & 4}$};
      \node at (2,0) (A3) {$\tableau{& 4 \\ 1 & 2 & 3}$};
      \node at (3,0) (B1) {$\tableau{& 4 \\ 1 & \st{2} & 3}$};
      \node at (4,0) (B2) {$\tableau{& 4 \\ 1 & 2 & \st{3}}$};
      \node at (5,0) (C1) {$\tableau{& 4 \\ 1 & \st{2} & \st{3}}$};
      \node at (6,0) (C2) {$\tableau{& 3 \\ 1 & \st{2} & \st{4}}$};
      \node at (7,0) (C3) {$\tableau{& 3 \\ 1 & 2 & \st{4}}$};
      \draw[thick,color=red   ,label={[above]{$d_{2}$}}](A1) -- (A2) ;
      \draw[thick,color=blue,label={[above]{$d_{3}$}}]  (A2) -- (A3) ;
      \draw[thick,color=red   ,label={[above]{$d_{2}$}}](B1.05) -- (B2.175) ;
      \draw[thick,color=blue,label={[below]{$d_{3}$}}]  (B1.355) -- (B2.185) ;
      \draw[thick,color=blue,label={[above]{$d_{3}$}}]  (C1) -- (C2) ;
      \draw[thick,color=red   ,label={[above]{$d_{2}$}}](C2) -- (C3) ;
    \end{tikzpicture}
    \caption{\label{fig:SST-deg}Dual equivalence for $\SST(3,1)$.}
  \end{center}
\end{figure}

Therefore \cite{Ass18}(Theorem~6.3) gives another proof of the nonnegativity of the Schur coefficients of Schur $P$-functions using dual equivalence.

\begin{theorem}[\cite{Ass18}]
  The involutions $\{d_i\}_{1<i<n}$ give a dual equivalence for signed standard tableaux. In particular, Schur $P$-functions are Schur positive.
  \label{thm:strong_pos}
\end{theorem}

%
\section{A queer involution}
%
\label{sec:queer}

We augment the dual equivalence structure on signed standard tableaux in Definition~\ref{def:deg_shifted} with an additional involution so that the graph on signed standard tableaux of a given shape is connected.

\begin{definition}
  The \newword{elementary queer involution}, denoted by $d_0$, acts on signed standard tableaux by toggling the mark on the entry with absolute value $2$. 
  \label{def:d_0}
\end{definition}

Given the row and column conditions on signed standard tableaux, the entry with absolute value $2$ must always lie in the first row and second column. In particular, it is never on the main diagonal, and so it may always come with or without a marking. Therefore $d_0$ is indeed a well-defined involution. For examples, see Fig.~\ref{fig:qdeg-31} and Fig.~\ref{fig:qdeg-4}.

\begin{figure}[ht]
  \begin{center}
    \begin{tikzpicture}[scale=1.6,
        label/.style={%
          postaction={ decorate,
            decoration={ markings, mark=at position 0.5 with \node #1;}}}]
      \node at (0,0) (A1) {$\tableau{& 3 \\ 1 & \st{2} & 4}$};
      \node at (1,0) (A2) {$\tableau{& 3 \\ 1 & 2 & 4}$};
      \node at (2,0) (A3) {$\tableau{& 4 \\ 1 & 2 & 3}$};
      \node at (3,0) (B1) {$\tableau{& 4 \\ 1 & \st{2} & 3}$};
      \node at (4,0) (B2) {$\tableau{& 4 \\ 1 & 2 & \st{3}}$};
      \node at (5,0) (C1) {$\tableau{& 4 \\ 1 & \st{2} & \st{3}}$};
      \node at (6,0) (C2) {$\tableau{& 3 \\ 1 & \st{2} & \st{4}}$};
      \node at (7,0) (C3) {$\tableau{& 3 \\ 1 & 2 & \st{4}}$};
      \draw[thick,color=red   ,label={[above]{$d_{2}$}}](A1.05) -- (A2.175) ;
      \draw[thick,color=blue,label={[above]{$d_{3}$}}]  (A2) -- (A3) ;
      \draw[thick,color=red   ,label={[above]{$d_{2}$}}](B1.05) -- (B2.175) ;
      \draw[thick,color=blue,label={[below]{$d_{3}$}}]  (B1.355) -- (B2.185) ;
      \draw[thick,color=blue,label={[above]{$d_{3}$}}]  (C1) -- (C2) ;
      \draw[thick,color=red   ,label={[above]{$d_{2}$}}](C2.05) -- (C3.175) ;
      \draw[thick,color=violet,label={[below]{$d_{0}$}}] (A1.355) -- (A2.185) ;
      \draw[thick,color=violet,label={[below]{$d_{0}$}}] (A3) -- (B1) ;
      \draw[thick,color=violet,label={[below]{$d_{0}$}}] (B2) -- (C1) ;
      \draw[thick,color=violet,label={[below]{$d_{0}$}}] (C2.355) -- (C3.185) ;
    \end{tikzpicture}
    \caption{\label{fig:qdeg-31}Queer dual equivalence for $\SST(3,1)$.}
  \end{center}
\end{figure}

\begin{definition}
  Two signed marked tableaux $S, T$ are \newword{queer dual equivalent} if one can be transformed into the other by applying a sequence of elementary dual equivalence and elementary queer involutions.
  \label{def:queer_equiv}
\end{definition}

\begin{figure}[ht]
  \begin{center}
    \begin{tikzpicture}[xscale=2.2,
        label/.style={%
          postaction={ decorate,
            decoration={ markings, mark=at position 0.5 with \node #1;}}}]
      \node at (1,1) (T1) {$\tableau{1 & 2 & 3 & 4}$};
      \node at (2,1) (T2) {$\tableau{1 & \st{2} & 3 & 4}$};
      \node at (3,1) (T3) {$\tableau{1 & 2 & \st{3} & 4}$};
      \node at (4,1) (T4) {$\tableau{1 & 2 & 3 & \st{4}}$};
      \node at (3,0) (B3) {$\tableau{1 & \st{2} & \st{3} & 4}$};
      \node at (4,0) (B4) {$\tableau{1 & \st{2} & 3 & \st{4}}$};
      \node at (5,0) (B5) {$\tableau{1 & 2 & \st{3} & \st{4}}$};
      \node at (6,0) (B6) {$\tableau{1 & \st{2} & \st{3} & \st{4}}$};
      \draw[thick,color=red   ,label={[above]{$d_{2}$}}](T2) -- (T3) ;
      \draw[thick,color=blue,label={[above]{$d_{3}$}}]  (T3) -- (T4) ;
      \draw[thick,color=blue,label={[above]{$d_{3}$}}]  (B3) -- (B4) ;
      \draw[thick,color=red   ,label={[above]{$d_{2}$}}](B4) -- (B5) ;
      \draw[thick,color=violet,label={[above]{$d_{0}$}}] (T1) -- (T2) ;
      \draw[thick,color=violet,label={[left]{$d_{0}$}}]  (T3) -- (B3) ;
      \draw[thick,color=violet,label={[right]{$d_{0}$}}] (T4) -- (B4) ;
      \draw[thick,color=violet,label={[above]{$d_{0}$}}] (B5) -- (B6) ;
    \end{tikzpicture}
    \caption{\label{fig:qdeg-4}Queer dual equivalence for $\SST(4)$.}
  \end{center}
\end{figure}

\begin{theorem}
  Two signed standard tableaux are queer dual equivalent if and only if they have the same shape.
  \label{thm:shape}
\end{theorem}

\begin{proof}
  Since none of the involutions changes the shape of the tableau, we need only show that any two signed standard tableaux of shape $\gamma$ are queer dual equivalent. We proceed by induction on the size of $\gamma$. For $|\gamma|=1$, the result is trivial. For $|\gamma|=2$, we must have $\gamma=(2)$, and the two signed tableaux of shape $(2)$ are related by $d_0$ as shown below,
  \[ \tableau{1 & 2} \stackrel{\displaystyle d_0}{\longleftrightarrow} \tableau{1 & \st{2}} . \]
  Thus we may assume $|\gamma|=n \geq 3$. Let $S, T \in \SST(\gamma)$. Consider the cells occupied by the entries with absolute value $n$ in both $S$ and $T$.

  \textbf{Case A:} Consider first the case when both cells are in the same position in $S$ and in $T$. If both cells are marked or if both cells are not marked, then the result follows by induction by considering the signed standard tableau obtained by removing these cells. Otherwise, we may assume $n$ occurs in $S$ in the same position as $\st{n}$ occurs in $T$. In particular, the cell, say $x$, is not on the main diagonal. 

  \emph{Subcase A1:} Suppose there is a nearest removable corner, say $y$, that lies strictly above and strictly left of $x$. In this case, let $S^{\prime} \in \SST(\gamma)$ be any tableau with (unmarked) $n$ in position $x$, (unmarked) $n-1$ in position $y$, and (unmarked) $n-2$ immediately below $y$ if $y$ is one column left of $x$ or immediately left of $x$ if not. Similarly, let $T^{\prime} \in \SST(\gamma)$ be any tableau with $\st{n}$ in position $x$, (unmarked) $n-1$ in position $y$, and $\st{n-2}$ immediately below $y$ if $y$ is one column left of $x$ or immediately left of $x$ if not. For example, see Fig.~\ref{fig:row-above}. Then $S$ is queer dual equivalent to $S^{\prime}$ and $T$ is queer dual equivalent to $T^{\prime}$ by the previous argument. Finally, $d_{n-1}$ acts on both $S^{\prime}$ and $T^{\prime}$ by Definition~\ref{def:deg_shifted}(4), resulting in (unmarked) $n$ in position $y$ in both $d_{n-1}(S^{\prime})$ and $d_{n-1}(T^{\prime})$. Thus these latter two are also queer dual equivalent, completing the queer dual equivalence of $S$ and $T$. 

  \begin{figure}[ht]
    \begin{center}
      \begin{tikzpicture}[scale=2,
          label/.style={%
            postaction={ decorate,
              decoration={ markings, mark=at position 0.5 with \node #1;}}}]
        \node at (0,0) (S) {$\tableau{& & 6 \\ & \ & 5 \\ \ & \ & \ & 7}$};
        \node at (1,0) (S1){$\tableau{& & 7 \\ & \ & 5 \\ \ & \ & \ & 6}$};
        \node at (3,0) (T1){$\tableau{& & 7 \\ & \ & \st{5} \\ \ & \ & \ & \st{6}}$};
        \node at (4,0) (T) {$\tableau{& & 6 \\ & \ & \st{5} \\ \ & \ & \ & \st{7}}$};
        \draw[thick,color=blue,label={[above]{$d_{6}$}}]  (S) -- (S1) ;
        \draw[thick,color=blue,label={[above]{induction}}]  (S1) -- (T1) ;
        \draw[thick,color=blue,label={[above]{$d_{6}$}}]  (T1) -- (T) ;
      \end{tikzpicture}
      \caption{\label{fig:row-above}Example of the queer dual equivalence when the cells $x$ are in the same position and $\gamma$ has a higher removable corner.}
    \end{center}
  \end{figure}
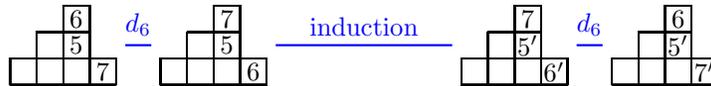

  \emph{Subcase A2:} Suppose there is nearest removable corner, say $y$, that lies strictly below and strictly right of $x$. In this case, let $S^{\prime} \in \SST(\gamma)$ be any tableau with (unmarked) $n$ in position $x$, (unmarked) $n-1$ in position $y$, and (unmarked) $n-2$ immediately left of $y$ if $y$ is one row below $x$ or immediately below $x$ if not. Similarly, let $T^{\prime} \in \SST(\gamma)$ be any tableau with $\st{n}$ in position $x$, (unmarked) $n-1$ in position $y$, and $\st{n-2}$ immediately left of $y$ if $y$ is one row below $x$ or immediately below $x$ if not. This is the transpose of the case depicted in Fig.~\ref{fig:row-above}. Then $S$ is queer dual equivalent to $S^{\prime}$ and $T$ is queer dual equivalent to $T^{\prime}$ by the previous argument, and $d_{n-1}$ acts on both $S^{\prime}$ and $T^{\prime}$ by Definition~\ref{def:deg_shifted}(4), resulting in (unmarked) $n$ in position $y$ in both $d_{n-1}(S^{\prime})$ and $d_{n-1}(T^{\prime})$. Thus these latter two are also queer dual equivalent, completing the queer dual equivalence of $S$ and $T$. 

  \emph{Subcase A3:} Suppose $x$ is the only removable corner and there are at least two cells to its left. Then let $S^{\prime} \in \SST(\gamma)$ be any tableau with (unmarked) $n$ in position $x$, $\st{n-1}$ immediately to its left, and (unmarked) $n-2$ immediately to its left. Then set $T^{\prime} = d_{n-1}(S^{\prime})$, which by Definition~\ref{def:deg_shifted}(3) will have $\st{n}$ in position $x$, (unmarked) $n-1$ immediately to its left, and (unmarked) $n-2$ immediately to its left. Then again $S$ and $\st{S}$ and also $T$ and $\st{T}$ are queer dual equivalent by virtue of the same marking on $n$ in the same cell, so $S$ is queer dual equivalent to $T$. 

  \emph{Subcase A4:} Otherwise, $x$ is the only removable corner and there is one cells to its left (since $x$ is not on the main diagonal) and a cell immediately below $x$ (since $n \geq 3$). Then let $S^{\prime} \in \SST(\gamma)$ be any tableau with (unmarked) $n$ in position $x$, $\st{n-1}$ immediately below $x$, and (unmarked) $n-2$ immediately left of $x$. Then set $T^{\prime} = d_{n-1}(S^{\prime})$, which by Definition~\ref{def:deg_shifted}(3) will have $\st{n}$ in position $x$, (unmarked) $n-1$ immediately below $x$, and (unmarked) $n-2$ immediately left of $x$. Then again $S$ and $\st{S}$ and also $T$ and $\st{T}$ are queer dual equivalent by virtue of the same marking on $n$ in the same cell, so $S$ is queer dual equivalent to $T$.

  \textbf{Case B:} Consider second the case when the cells are in different positions in $S$ and in $T$, say cell $x$ in $S$ and cell $y$ in $T$ and assume $y$ lies above and left of $x$. Then both $x$ and $y$ are removable corners of $\gamma$. In this case, let $\st{S} \in \SST(\gamma)$ be any tableau with $n$ in position $x$, $n-1$ in position $y$, and $n-2$ immediately below $y$ if $y$ is one column left of $x$ or immediately left of $x$ if not. If $n$ is marked in $S$, then mark each of $n-2,n-1,n$ in $\st{S}$; otherwise leave all three entries unmarked in $\st{S}$. Set $\st{T} = d_{n-1}(\st{S})$, which by Definition~\ref{def:deg_shifted}(4) will have $n$ in position $y$, $n-1$ in position $x$, and $n-2$ immediately below $y$ if $y$ is one column left of $x$ or immediately left of $x$ if not, with all marked or unmarked according to their state in $\st{S}$. Then $S$ is queer dual equivalent to $\st{S}$, and $T$ is queer dual equivalent to $\st{T}$ by case A, completing the proof that $S$ is queer dual equivalent to $T$.
\end{proof}

Theorem~\ref{thm:shape} allows us to shift our paradigm for Schur $P$-functions to
\begin{equation}
  P_{\gamma}(X) \ = \ \sum_{S \in [S_{\gamma}]} F_{\Des(S)}(X),
\label{e:queer-classes}
\end{equation}
where $[S_{\gamma}]$ denotes the queer dual equivalence class of some fixed $S_{\gamma} \in \SST(\gamma)$.

Note that, as with the dual equivalence classes, there is a natural choice for $S_{\gamma}$ for the queer dual equivalence class representative. Namely, if $S_{\gamma}$ is filled with the identity permutation from left to right beginning with the lowest row, then $S_{\gamma}$ is the unique tableau $S \in \SST(\gamma)$ such that $\Des(S) = \{\gamma_1, \gamma_1+\gamma_2, \ldots, \gamma_1+\cdots+\gamma_{\ell-1}\}$. 

The following corollary to the inductive proof of Theorem~\ref{thm:shape} will prove useful.

\begin{corollary}
  For $S,T \in \SST(\gamma)$ for some strict partition $\gamma$ of $n$, then there exists a sequence of indices $0 \leq i_1, i_2,\ldots,i_j< n$ at most two of which are $n-1$ such that $T = d_{i_j} \cdots d_{i_2} d_{i_1} (S)$. 
  \label{cor:two-edge}
\end{corollary}

%
\section{Queer dual equivalence}
%
\label{sec:queer-deg}

Again using the model of a set of objects $\mathcal{A}$ and a notion of descents, we can now give sufficient conditions to deduce that the quasisymmetric generating function for $\mathcal{A}$ is Schur $P$-positive.

\begin{definition}
  Let $\mathcal{A}$ be a finite set, and $\Des$ be a map from $\mathcal{A}$ to subsets of $[n-1]$. A \newword{queer dual equivalence for $(\mathcal{A},\Des)$} is dual equivalence $\{\psi_i\}_{1<i<n}$ together with a queer involution $\psi_0$ on $\mathcal{A}$ such that
  \renewcommand{\theenumi}{\roman{enumi}}
  \begin{enumerate}
  \item For $i = 1,2,3$ and $S \in \mathcal{A}$, there exists a strict partition $\gamma$ of $i+1$ such that
    \[ \sum_{U \in [S]_{\leq i}} F_{\Des(U)\cap[i]}(X) = P_{\gamma}(X), \]
    where $[S]_{\leq i}$ is the equivalence class generated by $\psi_h$ for $h\leq i$ ($h\neq 1$).
    
  \item For all $i > 3$ and all $S \in\mathcal{A}$, we have
    \begin{displaymath}
      \psi_{0} \psi_{i}(S) = \psi_{i} \psi_{0}(S).
    \end{displaymath}

  \item For $k < n$ and $S,T \in \mathcal{A}$ , if $T \in [S]_{\leq k}$, then there is a sequence of indices $0 \leq i_1, i_2,\ldots,i_j< k$ at most two of which are $k-1$ such that $T = \psi_{i_j} \cdots \psi_{i_2} \psi_{i_1} S$. 
    
  \end{enumerate}

  \label{def:queer-deg}
\end{definition}

\begin{remark}
  Taking $i=1$ in condition (i) of Definition~\ref{def:queer-deg} forces the queer involution $\psi_0$ to be fixed-point-free, since the only degree $2$ Schur $P$-function is $P_{(2)} = F_{\{\}} + F_{\{1\}}$, which has two terms. Moreover, we also must have $1 \in \Des(S)$ if and only if $1 \not\in\Des(\psi_0(S))$ for any $S \in \mathcal{A}$. 
  \label{rmk:fpf}
\end{remark}

Condition (iii) of Definition~\ref{def:queer-deg} parallels the non-local axiom $6$ of the equivalent characterization of dual equivalence stated in \cite{Ass15}(Definition~3.2). In the latter, axiom $6$ allows for at most one edge of maximum label. As proven in \cite{Ass15}(Theorem~4.2), axiom $6$ is equivalent to the \emph{local} condition (i) of Definition~\ref{def:deg}. In Section~\ref{sec:crystals}, we discuss efforts to give an equivalent local characterization for condition (iii) of Definition~\ref{def:queer-deg}.

To begin to justify our definition, we have the following simple check.

\begin{theorem}
  The queer dual equivalence involutions $d_0, \{d_i\}_{1<i<n}$ are a queer dual equivalence for signed standard tableaux.
  \label{thm:good-def}
\end{theorem}

\begin{proof}
  Given any signed standard tableau $S$, if we restrict $S$ to entries up to $i+1$, for $i\geq 3$, then the result is again a signed standard tableau, now of size $i+1$. Therefore by Theorem~\ref{thm:shape}, the generating function of the restricted equivalence class will be $P_{\gamma}$ for some $\gamma$ of size $i+1$, proving condition (i) of Definition~\ref{def:queer-deg}.

  For condition (ii), observe that the action of $d_i$ is local to entries with absolute value $i-1,i,i+1$, and when $i>3$ these entries do not include $2$. Since $d_0$ acts only on the $2$, it follows that $d_0$ and $d_i$ commute for $i>3$. 

  Finally, condition (iii) is precisely Corollary~\ref{cor:two-edge}.
\end{proof}

By Theorem~\ref{thm:good-def}, for $\gamma$ a strict partition of size $n$, we may define the \newword{standard queer dual equivalence graph} $\mathcal{H}_{\gamma}$ to be the graph on signed standard tableaux of shape $\gamma$ with an $i$-colored edge between $S$ and $d_i(S)$, for $i=0,2,\ldots,n-1$. Then the standard queer dual equivalence graphs are pairwise non-isomorphic and have no nontrivial automorphisms.

\begin{proposition}
  For strict partitions $\gamma,\delta$ of $n$, if $\theta: \mathcal{H}_{\gamma} \rightarrow \mathcal{H}_{\delta}$ is a bijection satisfying $\Des(\theta(S)) = \Des(S)$ and $d_i(\theta(S)) = \theta(d_i(S))$ for all $i=0,2,\ldots,n-1$, then $\delta=\gamma$ and $\theta$ is the identity map.
  \label{prop:noauto-noniso}
\end{proposition}

\begin{proof}
  Let $S_{\gamma}$ be the tableau obtained by filling the numbers 1 through $n$ into the rows of $\gamma$ from left to right, bottom to top, in which case $\Des(S_{\gamma}) = \{\gamma_1, \gamma_1 + \gamma_2, \ldots, \gamma_1+\cdots+\gamma_{\ell-1}\}$. For any standard tableau $T$ such that $\Des(S) = \Des(S_{\gamma})$, the numbers $1$ through $\gamma_1$, and also $\gamma_1 + 1$ through $\gamma_1 + \gamma_2$, and so on, must form horizontal strips. In particular, if $\Des(S) = \Des(S_{\gamma})$ for some $S$ of shape $\delta$, then $\gamma \leq \delta$ with equality if and only if $S = S_{\gamma}$.

  Suppose $\phi : \mathcal{H}_{\gamma} \rightarrow \mathcal{H}_{\delta}$ is a bijection preserving descent sets. Let $S_{\gamma}$ be as above for $\gamma$, and let $S_{\delta}$ be the corresponding tableau for $\delta$. Since $\Des(\phi(S_{\gamma})) = \Des(S_{\gamma})$,  we have $\gamma \leq \delta$. Conversely, since $\Des(\phi^{-1}(S_{\delta})) = \Des(S_{\delta})$, we have $\delta \leq \gamma$.  Therefore $\gamma=\delta$. Furthermore, $\phi(S_{\gamma}) = S_{\gamma}$. For $S \in \SST(\gamma)$ such that $d_i(S_{\gamma}) = S$ in $\mathcal{H}_{\gamma}$, we have $d_i(S_{\gamma}) = \phi(S)$ in $\mathcal{H}_{\delta}$, so $\phi(S) = S$. Extending this, every tableau connected to a fixed point by some sequence of involutions is also a fixed point for $\phi$, hence $\phi = \mathrm{id}$ on $\mathcal{H}_{\gamma}$ by Theorem~\ref{thm:shape}.
\end{proof}

Strengthening the uniqueness asserted by Proposition~\ref{prop:noauto-noniso}, the next three lemmas show that there is a unique queer dual equivalence that extends standard dual equivalence on signed standard tableaux of small size.

\begin{lemma}
  There exists a unique involution on $\SST(2)$ and a unique involution on $\SYT(2)\cup\SYT(1,1)$ that is a queer dual equivalence.
  \label{lem:unique-2}
\end{lemma}

\begin{proof}
  There are two signed standard tableaux of size $2$, and so since $\psi_0$ has no fixed points, by Remark~\ref{rmk:fpf}, the involution on these tableaux is the unique nontrivial one, shown in Fig.~\ref{fig:SST2}. Similarly, there is one standard Young tableaux of shape $(2)$ and one of shape $(1,1)$, and together their generating functions give $P_{(2)}$. Thus the involution in this case must connect the two tableaux, again shown in Fig.~\ref{fig:SST2},
\end{proof}

  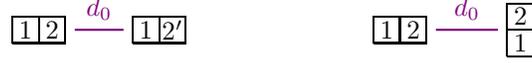
\begin{figure}[ht]
    \begin{center}
      \begin{tikzpicture}[xscale=1.6,
          label/.style={%
            postaction={ decorate,
              decoration={ markings, mark=at position 0.5 with \node #1;}}}]
        \node at (1,1) (S1) {$\tableau{1 & 2}$};
        \node at (2,1) (S2) {$\tableau{1 & \st{2}}$};
        \draw[thick,color=violet,label={[above]{$d_{0}$}}] (S1) -- (S2) ;
        \node at (4,1) (T1) {$\tableau{1 & 2}$};
        \node at (5,1) (T2) {$\tableau{2 \\ 1}$};
        \draw[thick,color=violet,label={[above]{$d_{0}$}}] (T1) -- (T2) ;
      \end{tikzpicture}
      \caption{\label{fig:SST2}The unique queer dual equivalences for $\SST(2)$ and for $\SYT(2)\cup\SYT(1,1)$.}
    \end{center}
  \end{figure}

\begin{lemma}
  Let $\gamma$ be a strict partition of $3$, and let $\Delta$ be the set of partitions defined by $P_{\gamma} = \sum_{\lambda \in \Delta} s_{\lambda}$. Then there exists a unique involution $\psi_0$ on $\bigsqcup_{\lambda\in\Delta} \SYT(\lambda)$ such that $\psi_0,d_2$ is a queer dual equivalence, where $d_2$ is the standard dual equivalence involution on standard Young tableaux.
  \label{lem:unique-3}
\end{lemma}

\begin{proof}
  Heeding Remark~\ref{rmk:fpf}, $\psi_0$ has no fixed points and toggles $1$ in and out of the decent set. For $\gamma=(2,1)$, we have $P_{(2,1)} = s_{(2,1)}$, so $\psi_0$ acts nontrivially to connect the two standard Young tableaux of shape $(2,1)$, as shown in Fig.~\ref{fig:SST3} (left).

  For $\gamma=(3)$, we have $P_{(3)} = s_{(3)} + s_{(2,1)} + s_{(1,1,1)}$, so we consider possible queer dual equivalences on $\SYT(3) \cup \SYT(2,1) \cup \SYT(1,1,1)$. The two possible pairings for $\psi_0$ that toggle $1$ from the descent set are the one shown in Fig.~\ref{fig:SST3} (right), and the one connecting the two tableaux of shape $(2,1)$ and then connecting the two remaining tableaux. The latter gives two disjoint queer dual equivalence classes, neither of which has generating function Schur $P$-positive, so the former holds.
\end{proof}

  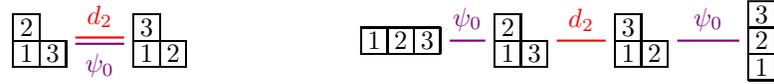
\begin{figure}[ht]
    \begin{center}
      \begin{tikzpicture}[xscale=1.6,
          label/.style={%
            postaction={ decorate,
              decoration={ markings, mark=at position 0.5 with \node #1;}}}]
        \node at (1,1) (T1) {$\tableau{1 & 2 & 3}$};
        \node at (2,1) (T2) {$\tableau{2 \\ 1 & 3}$};
        \node at (3,1) (T3) {$\tableau{3 \\ 1 & 2}$};
        \node at (4,1) (T4) {$\tableau{3 \\ 2 \\ 1}$};
        \node at (-2,1) (S2) {$\tableau{2 \\ 1 & 3}$};
        \node at (-1,1) (S3) {$\tableau{3 \\ 1 & 2}$};
        \draw[thick,color=red   ,label={[above]{$d_{2}$}}](T2) -- (T3) ;
        \draw[thick,color=violet,label={[above]{$\psi_{0}$}}] (T1) -- (T2) ;
        \draw[thick,color=violet,label={[above]{$\psi_{0}$}}] (T3) -- (T4) ;
        \draw[thick,color=red   ,label={[above]{$d_{2}$}}](S2.05) -- (S3.175) ;
        \draw[thick,color=violet,label={[below]{$\psi_{0}$}}] (S2.355) -- (S3.185) ;
      \end{tikzpicture}
      \caption{\label{fig:SST3}The unique queer dual equivalence for $\SYT(2,1)$ and for $\SYT(3) \cup \SYT(2,1) \cup \SYT(1,1,1)$.}
    \end{center}
  \end{figure}

\begin{lemma}
  Let $\gamma$ be a strict partition of $4$, and let $\Delta$ be the set of partitions defined by $P_{\gamma} = \sum_{\lambda \in \Delta} s_{\lambda}$. Then there exists a unique involution $\psi_0$ on $\bigsqcup_{\lambda\in\Delta} \SYT(\lambda)$ such that $\psi_0,d_2,d_3$ is a queer dual equivalence, where $d_2, d_3$ are the standard dual equivalence involutions on standard Young tableaux.
  \label{lem:unique-4}
\end{lemma}

\begin{proof}
  For the discussion to follow, we identify each tableau with its row reading word. Consider first $\gamma=(3,1)$, for which we have $P_{(3,1)} = s_{(3,1)} + s_{(2,2)} + s_{(2,1,1)}$, so we wish to construct a queer dual equivalence on
  \[ \SYT(3,1) \cup \SYT(2,2) \cup \SYT(2,1,1). \]
  By Lemma~\ref{lem:unique-3}, restricting to entries $1,2,3$ gives one of the two cases depicted in Fig.~\ref{fig:SST3}, giving three possibilities for $\psi_0$ acting on $4123$: $2413$, $2134$, or $4213$.
    
  Suppose $\psi_0(4123)=2134$, as depicted in Fig.~\ref{fig:SST4-bad}. Then restricting to entries $1,2,3$, we must have a component of the form in the right side of Fig.~\ref{fig:SST3}, and so we must have $\psi_0(3124)=3214$ as shown. Now there are two possibilities for $\psi_0(4213)$: either $4312$ or $3412$. The former results in two separate queer equivalence classes, so we must have the latter, which also forces $\psi_0(4312)=2413$, as shown with dashed lines in Fig.~\ref{fig:SST4-bad}. Restricting as in Definition~\ref{def:queer-deg} condition (i) for $i=2$, the queer dual equivalence class containing these last four tableaux has generating function $2 P_{(2,1)}$, contradicting that it is a queer dual equivalence. 

  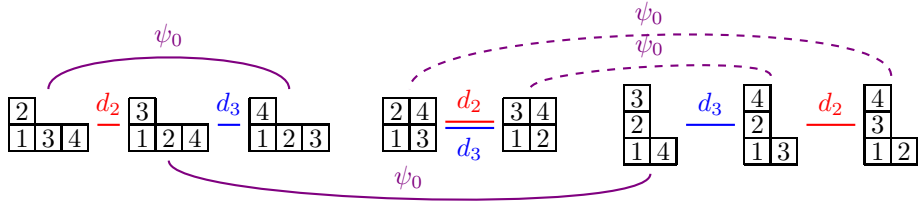
\begin{figure}[ht]
  \begin{center}
    \begin{tikzpicture}[scale=1.6,
        label/.style={%
          postaction={ decorate,
            decoration={ markings, mark=at position 0.5 with \node #1;}}}]
      \node (A1) at (0,1) {$\tableau{2 \\ 1 & 3 & 4}$};
      \node (A2) at (1,1) {$\tableau{3 \\ 1 & 2 & 4}$};
      \node (A3) at (2,1) {$\tableau{4 \\ 1 & 2 & 3}$};
      \node (B1) at (3,1) {$\tableau{2 & 4 \\ 1 & 3}$};
      \node (B2) at (4,1) {$\tableau{3 & 4 \\ 1 & 2}$};
      \node (C1) at (5,1) {$\tableau{3 \\ 2 \\ 1 & 4}$};
      \node (C2) at (6,1) {$\tableau{4 \\ 2 \\ 1 & 3}$};
      \node (C3) at (7,1) {$\tableau{4 \\ 3 \\ 1 & 2}$};
      \draw[thick,color=red   ,label={[above]{$d_{2}$}}](A1) -- (A2) ;
      \draw[thick,color=blue,label={[above]{$d_{3}$}}]  (A2) -- (A3) ;
      \draw[thick,color=red   ,label={[above]{$d_{2}$}}](B1.05) -- (B2.175) ;
      \draw[thick,color=blue,label={[below]{$d_{3}$}}]  (B1.355) -- (B2.185) ;
      \draw[thick,color=blue,label={[above]{$d_{3}$}}]  (C1) -- (C2) ;
      \draw[thick,color=red   ,label={[above]{$d_{2}$}}](C2) -- (C3) ;
      \draw[label={[above]{$\psi_{0}$}},thick,color=violet](A1) to[out=90,in=90, distance=\cellsize] (A3) ;
      \draw[label={[above]{$\psi_{0}$}},thick,color=violet](A2) to[out=-90,in=-90, distance=\cellsize] (C1) ;
      \draw[label={[above]{$\psi_{0}$}},thick,dashed,color=violet](C2) to[out=90,in=90, distance=0.5\cellsize] (B2) ;
      \draw[label={[above]{$\psi_{0}$}},thick,dashed,color=violet](C3) to[out=90,in=90, distance=1.5\cellsize] (B1) ;
    \end{tikzpicture}
    \caption{\label{fig:SST4-bad}A failed construction for a queer dual equivalence for $\SYT(3,1) \cup \SYT(2,2) \cup \SYT(2,1,1)$.}
  \end{center}
\end{figure}

  Similarly, if $\psi_0(4123)=4213$, then restricting to entries $1,2,3$, we must have a component of the form in the right side of Fig.~\ref{fig:SST3}, and so we must have $\psi_0(4312)=3214$. This case can be realized by conjugating each of the tableaux in Fig.~\ref{fig:SST4-bad}, and the same analysis reaches the same contradiction. Therefore we must have $\psi_0(4123)=2413$ as desired.

  Consider next the three possibilities for $\psi_0$ applied to $3214$: $3412$ (desired result), $4312$, or $3124$. The analysis for the three possibilities for $\psi_0(3214)$, namely $3412$, $4312$, or $3124$, is identical to the case for $\psi_0(4123)$ after the tableaux are conjugated, and so we conclude that we must have $\psi_0(3214) = 3412$, and we arrive at the situation depicted in Fig.~\ref{fig:SST4-good}. This leaves two possibilities for the remaining four tableaux, and the one not shown by dashed lines in Fig.~\ref{fig:SST4-good} again results in a restricted queer dual equivalence class with generating function $2 P_{(2,1)}$, leaving only the shown dashed lines as a possibility. Thus Fig.~\ref{fig:SST4-good} shows the unique queer dual equivalence for $\SYT(3,1) \cup \SYT(2,2) \cup \SYT(2,1,1)$.
  
  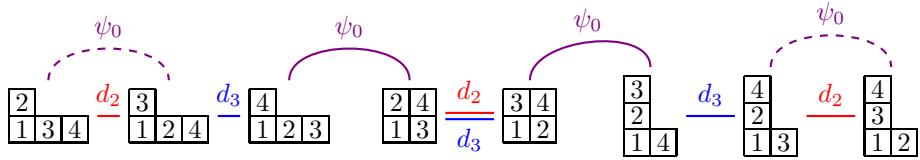
\begin{figure}[ht]
  \begin{center}
    \begin{tikzpicture}[scale=1.6,
        label/.style={%
          postaction={ decorate,
            decoration={ markings, mark=at position 0.5 with \node #1;}}}]
      \node (A1) at (0,1) {$\tableau{2 \\ 1 & 3 & 4}$};
      \node (A2) at (1,1) {$\tableau{3 \\ 1 & 2 & 4}$};
      \node (A3) at (2,1) {$\tableau{4 \\ 1 & 2 & 3}$};
      \node (B1) at (3,1) {$\tableau{2 & 4 \\ 1 & 3}$};
      \node (B2) at (4,1) {$\tableau{3 & 4 \\ 1 & 2}$};
      \node (C1) at (5,1) {$\tableau{3 \\ 2 \\ 1 & 4}$};
      \node (C2) at (6,1) {$\tableau{4 \\ 2 \\ 1 & 3}$};
      \node (C3) at (7,1) {$\tableau{4 \\ 3 \\ 1 & 2}$};
      \draw[thick,color=red   ,label={[above]{$d_{2}$}}](A1) -- (A2) ;
      \draw[thick,color=blue,label={[above]{$d_{3}$}}]  (A2) -- (A3) ;
      \draw[thick,color=red   ,label={[above]{$d_{2}$}}](B1.05) -- (B2.175) ;
      \draw[thick,color=blue,label={[below]{$d_{3}$}}]  (B1.355) -- (B2.185) ;
      \draw[thick,color=blue,label={[above]{$d_{3}$}}]  (C1) -- (C2) ;
      \draw[thick,color=red   ,label={[above]{$d_{2}$}}](C2) -- (C3) ;
      \draw[label={[above]{$\psi_{0}$}},thick,dashed,color=violet](A1) to[out=90,in=90, distance=\cellsize] (A2) ;
      \draw[label={[above]{$\psi_{0}$}},thick,dashed,color=violet](C2) to[out=90,in=90, distance=\cellsize] (C3) ;
      \draw[label={[above]{$\psi_{0}$}},thick,color=violet](A3) to[out=90,in=90, distance=\cellsize] (B1) ;
      \draw[label={[above]{$\psi_{0}$}},thick,color=violet](C1) to[out=90,in=90, distance=\cellsize] (B2) ;
    \end{tikzpicture}
    \caption{\label{fig:SST4-good}Constructing the unique queer dual equivalence for $\SYT(3,1) \cup \SYT(2,2) \cup \SYT(2,1,1)$.}
  \end{center}
\end{figure}

  The case of $\gamma = (4)$ has
  \[ P_{(4)}  = s_{(4)} + s_{(3,1)} + s_{(2,1,1)} + s_{(1,1,1,1)}, \]
  and so we consider possible queer dual equivalences for
  \[ \SYT(4) \cup \SYT(3,1) \cup \SYT(2,1,1) \cup \SYT(1,1,1,1). \]
  This case is handled by a parallel analysis. We omit the details but include the unique structure in Fig.~\ref{fig:SST4-good-2}, with solid lines corresponding to the initial cases to deduce and dashed lines corresponding to second cases to deduce, parallel to the argument for $\gamma=(3,1)$.
\end{proof}

\begin{figure}[ht]
  \begin{center}
    \begin{tikzpicture}[scale=1.6,
        label/.style={%
          postaction={ decorate,
            decoration={ markings, mark=at position 0.5 with \node #1;}}}]
      \node at (1,1) (T1) {$\tableau{1 & 2 & 3 & 4}$};
      \node at (2,1) (T2) {$\tableau{2 \\ 1 & 3 & 4}$};
      \node at (3,1) (T3) {$\tableau{3 \\ 1 & 2 & 4}$};
      \node at (4,1) (T4) {$\tableau{4 \\ 1 & 2 & 3}$};
      \node at (5,1) (B3) {$\tableau{3 \\ 2 \\ 1 & 4}$};
      \node at (6,1) (B4) {$\tableau{4 \\ 2 \\ 1 & 3}$};
      \node at (7,1) (B5) {$\tableau{4 \\ 3 \\ 1 & 2}$};
      \node at (8,1) (B6) {$\tableau{4 \\ 3 \\ 2 \\ 1}$};
      \draw[thick,color=red   ,label={[above]{$d_{2}$}}](T2) -- (T3) ;
      \draw[thick,color=blue,label={[above]{$d_{3}$}}]  (T3) -- (T4) ;
      \draw[thick,color=blue,label={[above]{$d_{3}$}}]  (B3) -- (B4) ;
      \draw[thick,color=red   ,label={[above]{$d_{2}$}}](B4) -- (B5) ;
      \draw[label={[above]{$\psi_{0}$}},thick,color=violet] (T1) to[out=90,in=90, distance=\cellsize] (T2) ;
      \draw[label={[above]{$\psi_{0}$}},thick,color=violet,dashed] (T3) to[out=90,in=90,distance=\cellsize] (B3) ;
      \draw[label={[above]{$\psi_{0}$}},thick,color=violet,dashed] (T4) to[out=-90,in=-90,distance=\cellsize] (B4) ;
      \draw[label={[above]{$\psi_{0}$}},thick,color=violet] (B5) to[out=90,in=90,distance=\cellsize] (B6) ;
    \end{tikzpicture}
    \caption{\label{fig:SST4-good-2}Constructing the unique queer dual equivalence for $\SYT(4) \cup \SYT(3,1) \cup \SYT(2,1,1) \cup \SYT(1,1,1,1)$.}
  \end{center}
\end{figure}
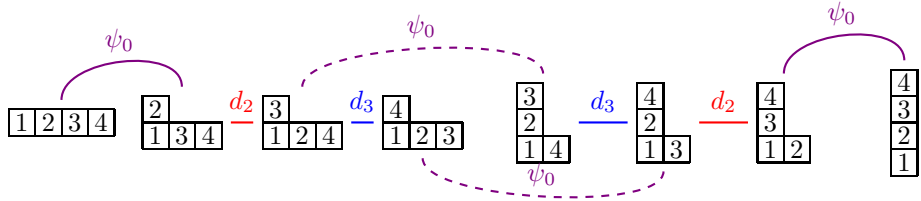

The following begins to show the restrictiveness of queer dual equivalence.

\begin{lemma}
  Let $\{\psi_i\}$ be a queer dual equivalence for $\mathcal{A}$ with descent map $\Des$. For $T \in \mathcal{A}$ and $i\geq 3$, we have $i\in\Des(T)$ if and only if $i\in\Des(\psi_0(T))$.
  \label{lem:des0}
\end{lemma}

\begin{proof}
  We proceed by induction on $i$. By condition (i) of Definition~\ref{def:queer-deg}, each restricted equivalence class for $\psi_0,\psi_2,\psi_3$ must be in descent-preserving bijection with signed standard tableaux. Therefore, by Lemma~\ref{lem:unique-4}, we have $3 \in \Des(T)$ if and only if $3 \in \Des(\psi_0(T))$, establishing the base case. 

  Assume the result for $3,4,\ldots,i-1$ and consider the action of $d_i$ on $T$, where now $i \geq 4$. If $T$ is a fixed point for $d_{i}$, then by the commutativity axiom for queer dual equivalence, we have $d_{i}(\psi_0(T)) = \psi_0(d_{i}(T)) = \psi_0(T)$, and so $\psi_0(T)$ is also a fixed point for $d_{i}$. This forces either both or neither $i-1,i$ in $\Des(T)$ and similarly for $\Des(\psi_0(T))$, so the result for $i-1$ forces the result for $i$. Alternatively, if $T$ is not a fixed point for $d_{i}$, then exactly one of $i-1,i$ is in $\Des(T)$, and $d_{i}$ toggles which is in and which is out, again the result for $i$ from the result for $i-1$. 
\end{proof}

We can now prove a stronger uniqueness result, namely that the queer dual equivalence involution in Definition~\ref{def:d_0} is the only involution that extends the dual equivalence involutions in Definition~\ref{def:deg_shifted} to an abstract queer dual equivalence.

\begin{theorem}
  For $\gamma$ a strict partition, the map $d_0$ of Definition~\ref{def:d_0} is the unique involution on $\SST(\gamma)$ for which $d_0$, $\{d_i\}_{1<i<n}$ define a queer dual equivalence involution, where $d_i$ for $i>1$ is given by Definition~\ref{def:deg_shifted}.  
\end{theorem}

\begin{proof}
  We proceed by induction on $n = |\gamma|$. For $n \leq 4$, there is a unique bijection between $\SST(\gamma)$ and $\bigsqcup_{\lambda\in\Delta} \SYT(\lambda)$, where $\Delta$ is the set of partitions defined by $P_{\gamma} = \sum_{\lambda \in \Delta} s_{\lambda}$, that intertwines the standard dual equivalences on $\SST$ and $\SYT$. Thus the base cases $n \leq 4$ are resolved by Lemmas~\ref{lem:unique-3} and \ref{lem:unique-4}. Assume then $n \geq 5$, and let $\psi_0$ be any involution that extends $\{d_i\}_{1<i<n}$ to a queer dual equivalence on $\SST(\gamma)$ for some strict partition $\gamma$ of $n$.

  By Lemma~\ref{lem:des0}, $\psi_0$ preserves descents at and above $3$. We claim the stronger statement that the positions and signs of entries $3,4,\ldots,n$ are preserved by $\psi_0$. This is true for the base cases by Lemmas~\ref{lem:unique-3} and \ref{lem:unique-4}. If $\psi_0$ changes the position of some $k<n$, then restricting the tableaux to entries up to $k$ would contradict the result for smaller values, so we may assume entries $3,4,\ldots,n-1$ are preserved by $\psi_0$ by induction. Therefore at worst $\psi_0$ toggles the sign on $n$. For this to occur, we must have $n$ not on the main diagonal (else it could not have a sign) and so some entry with absolute value $k$ immediately left of $n$ and in the same row. Then we must have $k\geq 3$ since the smallest entry in row $2$ or higher is $3$ and $n\geq 5$ in row $1$ must have at least three entries to its left. By Theorem~\ref{thm:shape}, we may apply some sequence of dual equivalence involution $d_{k+1},\ldots,d_{n-1}$ until $k$ is replaced by $n-1$. Since all of those involutions commute with $\psi_0$, the sign of $n$ must be toggled on the result as well, which toggles whether or not $n-1$ is a descent, contradicting Lemma~\eqref{lem:des0}. Thus the positions and signs of entries $3,4,\ldots,n$ are preserved by $\psi_0$, guaranteeing its uniqueness.
\end{proof}

%
\section{Schur $P$-positivity}
%

To show that the existence of a queer dual equivalence implies Schur $P$-positivity, we mimic the proof that dual equivalence implies Schur positivity (see \cite{Ass15}), beginning with the following lemma.

\begin{lemma}[\cite{Ass15}(Lemma~3.11)]
  Let $\{\varphi_i\}_{1<i<n+1}$ be a dual equivalence for $(\mathcal{A},\Des)$. If for $T \in \mathcal{A}$ and for some partition $\lambda$ of $n$, there is a bijection $f:[T]_{\leq n-1}\rightarrow\SYT(\lambda)$ such that for any $U\in[T]_{\leq n-1}$, we have
  \begin{itemize}
  \item $f ( \varphi_i ( U ) ) = d_i ( f ( U ) )$ for $i=2,\ldots,n-1$, and
  \item $\Des(f(U)) \cap [n-1] = \Des(U) \cap [n-1]$,
  \end{itemize}
  then there exists a unique partition $\mu$ of $n+1$ with $\lambda \subset \mu$ such that for any tableau $S \in \SYT(\mu)$ that restricts to $\SYT(\lambda)$, we have $\Des(f^{-1}(S|_{\leq n}))=\Des(S)$.
\end{lemma}

In the unshifted case, each pair of partitions $\lambda\subset\mu$ with $|\mu|-|\lambda|=1$ has one natural inclusion map $\SYT(\lambda)\subset\SYT(\mu)$. In the shifted case, a pair of strict partitions $\gamma \subset \delta$ with $|\delta|-|\gamma|=1$ can have two natural inclusion maps $\SST(\gamma)\subset\SST(\delta)$, based on whether we regard the letter in the unique cell in the difference of diagrams to have a marking or not. For example, Fig.~\ref{fig:imbed-41} shows the queer dual equivalence for $\SST(4,1)$ with cells containing entry $5$ or $\st{5}$ highlighted. Notice there are two embeddings of $\SST(3,1)$ into $\SST(4,1)$, by appending a cell with entry $5$ or $\st{5}$, and only one embedding of $\SST(4)$ into $\SST(4,1)$. 

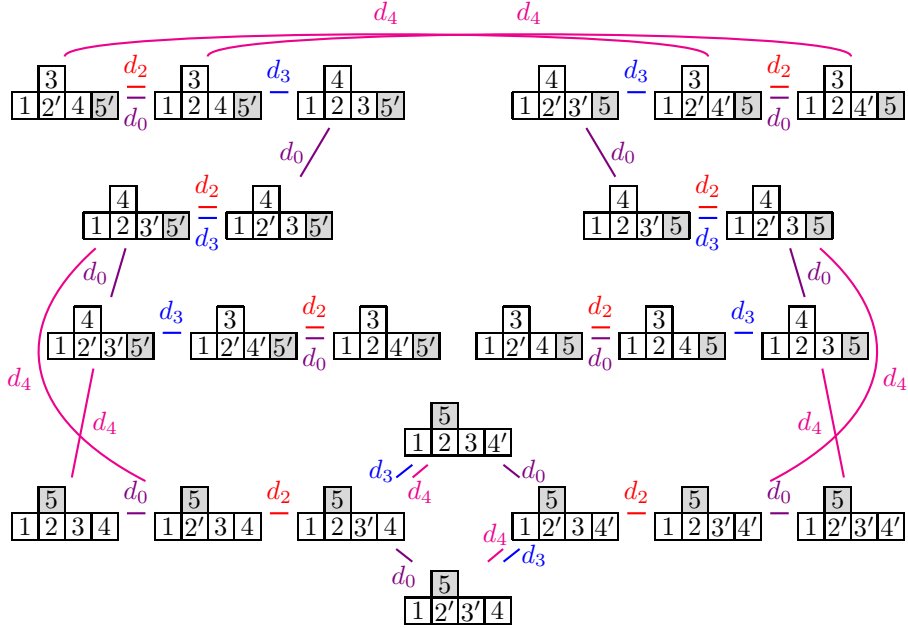
\begin{figure}[ht]
  \begin{center}
    \begin{tikzpicture}[xscale=1.9,yscale=1.6,
        label/.style={%
          postaction={ decorate,
            decoration={ markings, mark=at position 0.5 with \node #1;}}}]
      \node at (0,2)   (A1) {$\tableau{& 3 \\ 1 & \st{2} & 4 & \grayify{\st{5}}}$};
      \node at (1,2)   (A2) {$\tableau{& 3 \\ 1 & 2 & 4 & \grayify{\st{5}}}$};
      \node at (2,2)   (A3) {$\tableau{& 4 \\ 1 & 2 & 3 & \grayify{\st{5}}}$};
      \node at (1.5,1) (B1) {$\tableau{& 4 \\ 1 & \st{2} & 3 & \grayify{\st{5}}}$};
      \node at (0.5,1) (B2) {$\tableau{& 4 \\ 1 & 2 & \st{3} & \grayify{\st{5}}}$};
      \node at (0.25,0)   (C1) {$\tableau{& 4 \\ 1 & \st{2} & \st{3} & \grayify{\st{5}}}$};
      \node at (1.25,0)   (C2) {$\tableau{& 3 \\ 1 & \st{2} & \st{4} & \grayify{\st{5}}}$};
      \node at (2.25,0)   (C3) {$\tableau{& 3 \\ 1 & 2 & \st{4} & \grayify{\st{5}}}$};
      \node at (3.25,0) (a1) {$\tableau{& 3 \\ 1 & \st{2} & 4 & \grayify{5}}$};
      \node at (4.25,0) (a2) {$\tableau{& 3 \\ 1 & 2 & 4 & \grayify{5}}$};
      \node at (5.25,0) (a3) {$\tableau{& 4 \\ 1 & 2 & 3 & \grayify{5}}$};
      \node at (5,1) (b1) {$\tableau{& 4 \\ 1 & \st{2} & 3 & \grayify{5}}$};
      \node at (4,1) (b2) {$\tableau{& 4 \\ 1 & 2 & \st{3} & \grayify{5}}$};
      \node at (3.5,2) (c1) {$\tableau{& 4 \\ 1 & \st{2} & \st{3} & \grayify{5}}$};
      \node at (4.5,2) (c2) {$\tableau{& 3 \\ 1 & \st{2} & \st{4} & \grayify{5}}$};
      \node at (5.5,2) (c3) {$\tableau{& 3 \\ 1 & 2 & \st{4} & \grayify{5}}$};
      \node at (0,-1.5) (T1) {$\tableau{& \grayify{5} \\ 1 & 2 & 3 & 4}$};
      \node at (1,-1.5) (T2) {$\tableau{& \grayify{5} \\ 1 & \st{2} & 3 & 4}$};
      \node at (2,-1.5) (T3) {$\tableau{& \grayify{5} \\ 1 & 2 & \st{3} & 4}$};
      \node at (2.75,-0.8) (T4) {$\tableau{& \grayify{5} \\ 1 & 2 & 3 & \st{4}}$};
      \node at (2.75,-2.2) (U3) {$\tableau{& \grayify{5} \\ 1 & \st{2} & \st{3} & 4}$};
      \node at (3.5,-1.5) (U4) {$\tableau{& \grayify{5} \\ 1 & \st{2} & 3 & \st{4}}$};
      \node at (4.5,-1.5) (U5) {$\tableau{& \grayify{5} \\ 1 & 2 & \st{3} & \st{4}}$};
      \node at (5.5,-1.5) (U6) {$\tableau{& \grayify{5} \\ 1 & \st{2} & \st{3} & \st{4}}$};
      \draw[thick,color=red   ,label={[above]{$d_{2}$}}] (T2) -- (T3) ;
      \draw[thick,color=blue,label={[left ]{$d_{3}$}}]   (T3) -- (T4) ;
      \draw[thick,color=blue,label={[right]{$d_{3}$}}]   (U3) -- (U4) ;
      \draw[thick,color=red   ,label={[above]{$d_{2}$}}] (U4) -- (U5) ;
      \draw[thick,color=violet,label={[above]{$d_{0}$}}] (T1) -- (T2) ;
      \draw[thick,color=violet,label={[below]{$d_{0}$}}] (T3) -- (U3) ;
      \draw[thick,color=violet,label={[right]{$d_{0}$}}] (T4) -- (U4) ;
      \draw[thick,color=violet,label={[above]{$d_{0}$}}] (U5) -- (U6) ;
      \draw[thick,color=red   ,label={[above]{$d_{2}$}}] (A1.05) -- (A2.175) ;
      \draw[thick,color=blue  ,label={[above]{$d_{3}$}}] (A2) -- (A3) ;
      \draw[thick,color=red   ,label={[above]{$d_{2}$}}] (B2.05) -- (B1.175) ;
      \draw[thick,color=blue  ,label={[below]{$d_{3}$}}] (B2.355)-- (B1.185) ;
      \draw[thick,color=blue  ,label={[above]{$d_{3}$}}] (C1) -- (C2) ;
      \draw[thick,color=red   ,label={[above]{$d_{2}$}}] (C2.05) -- (C3.175) ;
      \draw[thick,color=violet,label={[below]{$d_{0}$}}] (A1.355) -- (A2.185) ;
      \draw[thick,color=violet,label={[left ]{$d_{0}$}}] (A3) -- (B1) ;
      \draw[thick,color=violet,label={[left ]{$d_{0}$}}] (B2) -- (C1) ;
      \draw[thick,color=violet,label={[below]{$d_{0}$}}] (C2.355) -- (C3.185) ;
      \draw[thick,color=red   ,label={[above]{$d_{2}$}}] (a1.05) -- (a2.175) ;
      \draw[thick,color=blue  ,label={[above]{$d_{3}$}}] (a2) -- (a3) ;
      \draw[thick,color=red   ,label={[above]{$d_{2}$}}] (b2.05) -- (b1.175) ;
      \draw[thick,color=blue  ,label={[below]{$d_{3}$}}] (b2.355)-- (b1.185) ;
      \draw[thick,color=blue  ,label={[above]{$d_{3}$}}] (c1) -- (c2) ;
      \draw[thick,color=red   ,label={[above]{$d_{2}$}}] (c2.05) -- (c3.175) ;
      \draw[thick,color=violet,label={[below]{$d_{0}$}}] (a1.355) -- (a2.185) ;
      \draw[thick,color=violet,label={[right]{$d_{0}$}}] (a3) -- (b1) ;
      \draw[thick,color=violet,label={[right]{$d_{0}$}}] (b2) -- (c1) ;
      \draw[thick,color=violet,label={[below]{$d_{0}$}}] (c2.355) -- (c3.185) ;
      \draw[thick,color=magenta ,label={[above]{$d_{4}$}}] (A1) to[out=90 ,in=90 , distance=0.7\cellsize] (c2) ;
      \draw[thick,color=magenta ,label={[above]{$d_{4}$}}] (A2) to[out=90 ,in=90 , distance=0.7\cellsize] (c3) ;
      \draw[thick,color=magenta ,label={[right]{$d_{4}$}}] (C1) -- (T1) ;
      \draw[thick,color=magenta ,label={[left ]{$d_{4}$}}] (B2) to[out=225,in=150, distance=2.75\cellsize] (T2) ;
      \draw[thick,color=magenta ,label={[right]{$d_{4}$}}] (b1) to[out=315,in=30 , distance=2.75\cellsize] (U5) ;
      \draw[thick,color=magenta ,label={[left ]{$d_{4}$}}] (a3) -- (U6) ;
      \draw[thick,color=magenta ,label={[below]{$d_{4}$}}] (T3.30) -- (T4.230) ;
      \draw[thick,color=magenta ,label={[above]{$d_{4}$}}] (U3.50) -- (U4.210) ;
    \end{tikzpicture}
    \caption{\label{fig:imbed-41}Queer dual equivalence for $\SST(4,1)$.}
  \end{center}
\end{figure}

For $\gamma$ a strict partition, the \newword{symmetric diagram} for $\gamma$, denoted by $\mathrm{sym}(\gamma)$, is
\begin{equation}
  \mathrm{sym}(\gamma)_i = \left\{ \begin{array}{rl}
    \gamma_i + i - 1 & \mbox{if $i \leq \ell(\gamma)$,} \\
    \#\{j \mid \gamma_j + j - 1 \geq i \} & \mbox{otherwise}.
  \end{array} \right.
  \label{e:sym}
\end{equation}
That is, the Young diagram for $\mathrm{sym}(\gamma)$ is the shifted diagram for $\gamma$ together with its transpose identified along the main diagonal; see Figure~\ref{fig:symmetric}. 

\begin{figure}[ht]
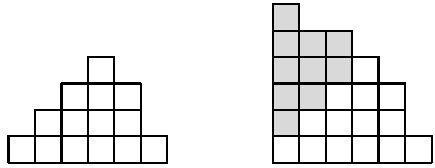

  \begin{displaymath}
    \tableau{\\ \\  & & & \ \\ &  & \ & \ & \ \\ & \ & \ & \ & \ \\ \ & \ & \ & \ & \ & \ }
    \hspace{4em}
    \tableau{ \cb \\ \cb & \cb & \cb \\ \cb & \cb & \cb & \ \\ \cb & \cb & \ & \ & \ \\ \cb & \ & \ & \ & \ \\ \ & \ & \ & \ & \ & \ }
  \end{displaymath}
  \caption{\label{fig:symmetric}The shifted (left) and symmetric (right) diagrams for $(6,4,3,1)$.}
\end{figure}

With this notation in mind, we have the following shifted analog.

\begin{lemma}
  Let $\{\psi_i\}_{i=0,2,\ldots,n}$ be a queer dual equivalence for $(\mathcal{A},\Des)$. If for $T \in \mathcal{A}$ and for some strict partition $\gamma$ of $n$, there is a bijection $f:[T]_{\leq n-1}\rightarrow\SST(\gamma)$ such that for any $U\in[T]_{\leq n-1}$, we have
  \begin{itemize}
  \item $f ( \psi_i ( U ) ) = d_i ( f ( U ) )$ for $i=0,2,\ldots,n-1$, and
  \item $\Des(f(U)) \cap [n-1] = \Des(U) \cap [n-1]$,
  \end{itemize}
  then there exists a unique strict partition $\delta$ of $n+1$ with $\gamma \subset \delta$ and an embedding of $\SST(\gamma)$ into $\SST(\delta)$ such that for any tableau $S \in \SST(\gamma)$ regarded as an element of $\SST(\delta)$, we have $\Des(f(U))=\Des(U)$ for every $U\in[T]_{\leq n-1}$. 
  \label{lem:extend-signs}
\end{lemma}

\begin{proof}
  We may regard $\SST(\gamma)$ as \emph{partial} fillings of $\SST(\mathrm{sym}(\gamma))$ obtained by leaving unmarked entries in place and reflecting each marked entry across the main diagonal. In so doing, dual equivalence on signed standard tableaux better resembles dual equivalence on standard Young tableaux in that $d_{n-1}$ acts by cases (3) or (4) with $n-2$ on diagonal $b$ in Definition~\ref{def:deg_shifted}, it does so by swapping $n$ and $n-1$ and, in case (3), shifting the landing spots one cell left/right or up/down.

  An \newword{addable cell} for a Young (resp. shifted) diagram is a cell not in the diagram such that adding it results in a Young (resp. shifted) diagram. Similarly, a \newword{removable cell} for a Young (resp. shifted) diagram is a cell in the diagram such that removing it results in a Young (resp. shifted) diagram. Notice that each addable/removable cell for a shifted diagram not on the main diagonal corresponds to exactly two addable/removable cells of its symmetric Young diagram, one above and one below the main diagonal. 

  Let $\mathcal{T} = [T]_{\leq n-1}$. Parallel to the proof of \cite{Ass15}(Lemma~3.11), we may use the bijection $f : \mathcal{T} \rightarrow \SST(\gamma)$ to identify uniquely each restricted class $[U]_{\leq n-2}$ for $U\in \mathcal{T}$ with a removable cell of $\mathrm{sym}(\gamma)$, where the identified cell is either the one containing $n$ or the reflection across the main diagonal of the one containing $\st{n}$, whichever case applies. As in \cite{Ass15}(Lemma~3.11), we have the property that $\Des\cap\{n\}$ is constant on each of these restricted classes $[U]_{\leq n-2}$. Moreover, if $U,V\in\mathcal{T}$ such that the cell for $[U]_{\leq n-2}$ lies northwest of the cell for $[V]_{\leq n-2}$, then \cite{Ass15}(Lemma~3.11) shows that (a) if $n \in \Des(U)$, then $n \in \Des(V)$, and (b) if $n \not\in \Des(V)$, then $n \not\in \Des(U)$. This is illustrated in Fig.~\ref{fig:monotone} for $\gamma=(6,4,3,1)$. 

  \begin{figure}[ht]
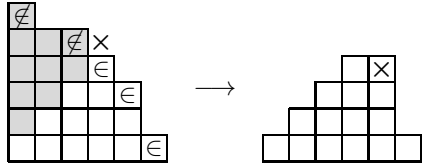

    \begin{displaymath}
      \tableau{ \grayify{\not\in} \\ \cb & \cb & \grayify{\not\in} \\ \cb & \cb & \cb & \in \\ \cb & \cb & \ & \ & \in \\ \cb & \ & \ & \ & \ \\ \ & \ & \ & \ & \ & \in } \hspace{-2.5\cellsize}\raisebox{-0.75\cellsize}{\makebox[0pt]{$\boldsymbol{\times}$}}
      \hspace{3.5em} \raisebox{-2.5\cellsize}{$\longrightarrow$} \hspace{1em}
      \tableau{\\ \\  & & & \ & \boldsymbol{\times}\\ &  & \ & \ & \ \\ & \ & \ & \ & \ \\ \ & \ & \ & \ & \ & \ }
    \end{displaymath}
    \caption{\label{fig:monotone}Illustration of the descent pattern of $n$ determining the partition $\delta \supset (6,4,3,1)$.}
  \end{figure}

  In particular, there is a unique addable cell of $\mathrm{sym}(\gamma)$ such that all removable cells above this have $n \not\in \Des$ and all removable cells below this have $n\in\Des$. Therefore we may add a cell containing $n+1$ in this position. Then $\delta$ is the resulting strict shape, and the embedding inserts $n+1$ if the addable cell is weakly below the main diagonal and $\st{n+1}$ if it is above.   
\end{proof}

The following result, analogous to \cite{Ass15}(Theorem~3.13), is the key to establishing Schur $P$-positivity by induction. It shows that we can lift the bijection from $\SST(\gamma)$ to $\SST(\delta)$, where $\delta$ is the partition found in Lemma~\ref{lem:extend-signs}.

\begin{theorem}
  Suppose $\{\psi_i\}_{i=0,2,\ldots,n}$ is a queer dual equivalence for $(\mathcal{A},\Des)$ such that for each restricted equivalence class, there is a bijection $f : [T]_{\leq n-1} \rightarrow \SST(\gamma)$ for some strict partition $\gamma$ of $n$ such that for all $U \in \mathcal{A}$, we have 
  \begin{itemize}
  \item $f ( \psi_i ( U ) ) = d_i ( f ( U ) )$ for $i=0,2,\ldots,n-1$, and
  \item $\Des(f(U)) \cap [n-1] = \Des(U) \cap [n-1]$.
  \end{itemize}
  Then $f$ extends to a descent-preserving map from $\mathcal{A}$ to $\SST(\delta)$ for some unique strict partition $\delta$ of $n+1$ such that $f ( \psi_i ( U ) ) = d_i ( f ( U ) )$ for $i=0,2,\ldots,n$.
  \label{thm:cover}
\end{theorem}

\begin{proof}
  The result follows for $n\leq 3$ by condition (i) of Definition~\ref{def:queer-deg}, so we may proceed by induction assuming $n \geq 4$. By Lemma~\ref{lem:extend-signs}, each bijection $f : [T]_{\leq n-1} \rightarrow \SST(\gamma)$ lifts to a unique descent-preserving injection to $\SST(\delta)$ for a unique strict partition $\delta \supset \gamma$. We must show $\delta$ is the same for each restricted equivalence class and that this lifted map satisfies $f \circ \psi_{n} = d_n \circ f$. It is enough to show this for $[T]_{\leq n-1}$ and $[\psi_n(T)]_{\leq n-1}$. However, \cite{Ass15}(Theorem~3.13) proves exactly this when we further restrict each class to avoid use of the queer involution $\psi_0$. Since the preservation of descents and intertwining of dual equivalence involutions holds for each of the restricted subclasses, it also holds for the union.
\end{proof}

Note that, as is the case for dual equivalence on Young tableaux, Theorem~\ref{thm:cover} does not state that the lifted map is a \emph{bijection}. It necessarily will be surjective, but injectivity can fail for shifted tableaux in the same way that it can for Young tableaux. The smallest example for which this can occur is for the strict partition $(4,1)$, as illustrated in Fig.~\ref{fig:cover} where we make a double cover of $\SST(4,1)$. Notice this fails condition (iii) of Definition~\ref{def:queer-deg} and so is not a queer dual equivalence.

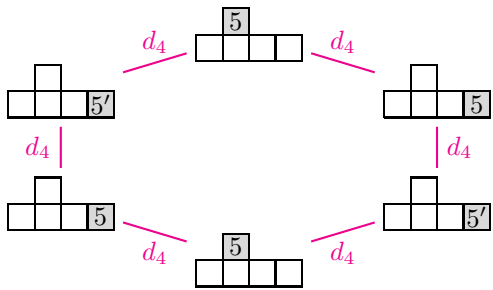
\begin{figure}[ht]
  \begin{center}
    \begin{tikzpicture}[xscale=2.5,yscale=1.5,
        label/.style={%
          postaction={ decorate,
            decoration={ markings, mark=at position 0.5 with \node #1;}}}]
      \node at (1,2.5) (T2) {$\tableau{& \grayify{5} \\ \ & \ & \ & \ }$};
      \node at (2,2) (a2) {$\tableau{& \ \\ \ & \ & \ & \grayify{5}}$};
      \node at (0,2) (A2) {$\tableau{& \ \\ \ & \ & \ & \grayify{\st{5}}}$};
      \node at (2,1) (A1) {$\tableau{& \ \\ \ & \ & \ & \grayify{\st{5}}}$};
      \node at (0,1) (a1) {$\tableau{& \ \\ \ & \ & \ & \grayify{5}}$};
      \node at (1,0.5) (T1) {$\tableau{& \grayify{5} \\ \ & \ & \ & \ }$};
      \draw[thick,color=magenta ,label={[below]{$d_{4}$}}] (T1) -- (a1) ;
      \draw[thick,color=magenta ,label={[below]{$d_{4}$}}] (T1) -- (A1) ;
      \draw[thick,color=magenta ,label={[left ]{$d_{4}$}}] (a1) -- (A2) ;
      \draw[thick,color=magenta ,label={[right]{$d_{4}$}}] (A1) -- (a2) ;
      \draw[thick,color=magenta ,label={[above]{$d_{4}$}}] (A2) -- (T2) ;
      \draw[thick,color=magenta ,label={[above]{$d_{4}$}}] (a2) -- (T2) ;
    \end{tikzpicture}
    \caption{\label{fig:cover}A double cover of $\SST(4,1)$, where each vertex represents a restricted queer dual equivalence class for $d_0,d_2,d_3$.}
  \end{center}
\end{figure}

The following analog of Theorem~\ref{thm:deg} shows that the queer dual equivalence on signed standard tableaux is the only queer dual equivalence for any set of objects.

\begin{theorem}
  If $\{\psi_i\}$ is a queer dual equivalence for $(\mathcal{A},\Des)$, then
  \begin{equation}
    \sum_{S \in [U]} F_{\Des(S)}(X) \ = \ P_{\delta}(X)
  \end{equation}
  for any $U \in \mathcal{A}$ and for some strict partition $\delta$. In particular, the fundamental quasisymmetric generating function for $\mathcal{A}$ is Schur $P$-positive.
  \label{thm:queer-deg}
\end{theorem}

\begin{proof}
  We proceed by induction on $n$, noting that condition (i) of Definition~\ref{def:queer-deg} makes the result immediate for $n \leq 4$. Assume then that $n \geq 5$ and the result holds for $n-1$. For simplicity, we may assume that $\mathcal{A}$ consists of a single queer dual equivalence class. Each restricted equivalence class for $\psi_0,\psi_2,\ldots,\psi_{n-2}$ also satisfies Definition~\ref{def:queer-deg} for $\Des$ restricted to $[n-2]$, so by induction, for each restricted equivalence class there is a bijection to signed standard tableaux. Therefore Theorem~\ref{thm:cover} applies and this bijection lifts to a descent-preserving map from $\mathcal{A}$ to $\SST(\delta)$ for some strict partition $\delta$ that intertwines $\psi_i$ with $d_i$. We must show this lift, which we call $f$, is a bijection.

  For surjectivity, take any $T \in \mathcal{A}$ and any $S\in\SST(\delta)$. Then $f(T)\in\SST(\delta)$, so by Theorem~\ref{thm:shape}, there exists a sequence of queer dual equivalence involutions taking $f(T)$ to $S$,, say $S = d_{i_k} \cdots d_{i_1} (f(T))$. By the intertwining property, we have $f( \psi_{i_k} \cdots \psi_{i_1} (T) ) = d_{i_k} \cdots d_{i_1} (f(T)) = S$. Thus $S$ is in the image of $f$.

  For injectivity, we will show that for each strict partition $\gamma \subset \delta$, there is a unique restricted equivalence class $[T]_{\leq n-2}$ such that $f([T]_{\leq n-2}) = \SST(\gamma)$ for each embedding of $\SST(\gamma)$ into $\SST(\delta)$. Suppose $U\in\mathcal{A}$ such that $f([U]_{\leq n-2}) = \SST(\gamma)$ for the same embedding. Then condition (iii) of Definition~\ref{def:queer-deg} ensures that there is there is a sequence of indices $0 \leq i_1, i_2,\ldots,i_j< n$ at most two of which are $n-1$ such that $T = \psi_{i_j} \cdots \psi_{i_2} \psi_{i_1} U$. This means there exists some $V\in\mathcal{A}$ such that both $T$ and $U$ are reachable by applying at most one $\psi_{n-1}$. However, since $f$ intertwines $\psi_i$ for all $i$, the restricted class $[V]_{\leq n-2}$ is connected via $\psi_{n-1}$ to a unique restricted class that maps to $\SST(\epsilon)$ for each $\epsilon\subset\delta$ (together with an embedding). Therefore $U \in [T]_{\leq n-2}$, and so $f$ is injective as well.

  Thus we have a descent-preserving bijection from $\mathcal{A}$ to $\SST(\delta)$. Taking generating functions, the theorem follows.
\end{proof}

Therefore queer dual equivalence gives a universal method for establishing Schur $P$-positivity of a quasisymmetric function expressed nonnegatively in the fundamental basis.

%
\section{Products of Schur $P$-functions}
%

Stembridge \cite{Ste89} extended shifted insertion \cite{Sag87,Wor84} in his study of projective representations of the symmetric group to prove that Schur $P$-functions have nonnegative structure constants.

\begin{theorem}[\cite{Ste89}]
  For $\gamma,\delta,\varepsilon$ strict partitions, there exist nonnegative integers $f_{\gamma,\delta}^{\varepsilon}$ such that
    \begin{equation}
      P_{\gamma} (X) P_{\delta} (X) \ = \ \sum_{\varepsilon} f_{\gamma,\delta}^{\varepsilon} P_{\varepsilon}(X) .
      \label{e:struct_P}
    \end{equation}
\label{thm:struct_P}
\end{theorem}

Define Schur $Q$ functions in terms of Schur $P$-functions by 
\begin{equation}
  Q_{\gamma} (X) \ = \ 2^{\ell(\gamma)} P_{\gamma}(X) .
\label{e:schur_Q}
\end{equation}
Then Schur $Q$ and $P$-functions form dual bases and the operation of skewing is adjoint to multiplication
\cite{Mac95}, the integers $f_{\gamma,\delta}^{\varepsilon}$ also satisfy
\begin{equation}
  Q_{\varepsilon/\gamma}(X) = \sum_{\delta} f_{\gamma,\delta}^{\varepsilon} Q_{\delta}(X). 
  \label{e:struct_Q}
\end{equation}
Shifted dual equivalence \cite{Ass18} gives another proof of positivity of Schur $P$-structure constants in Eq.~\eqref{e:struct_P} by instead considering the formulation for skew Schur $Q$ functions in Eq.~\eqref{e:struct_Q}. This happens because shifted dual equivalence is more naturally about Schur $Q$ functions than Schur $P$-functions. In contrast, we use queer dual equivalence to give a direct formula for Eq.~\eqref{e:struct_P}.

To begin, given strict partitions $\gamma,\delta$, we consider the concatenated shape $\gamma \otimes \delta$ with $\gamma$ written left of $\delta$ and rows aligned at the bottom. A \newword{signed standard tableau of shape $\gamma \otimes \delta$} is a bijective filling of $\gamma \otimes \delta$ with a signed permutation such that no signed entries appear on \emph{either} main diagonal; e.g. see Fig.~\ref{fig:tensor}. We denote signed standard tableaux of concatenated shape $\gamma \otimes \delta$ by $S\otimes T$, where $S$ has shape $\gamma$ and $T$ has shape $\delta$. Note neither $S$ nor $T$ is a signed standard tableau. 

\begin{figure}[ht]
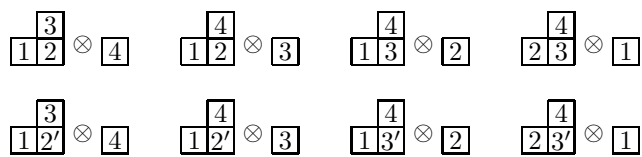

  \begin{center}
    \begin{displaymath}
      \begin{array}{c@{\hskip 2\cellsize}c@{\hskip 2\cellsize}c@{\hskip 2\cellsize}c}
      \tableau{& 3 \\ 1 & 2     }\ \raisebox{-0.5\cellsize}{$\otimes$}\ \tableau{ \\ 4} &
      \tableau{& 4 \\ 1 & 2     }\ \raisebox{-0.5\cellsize}{$\otimes$}\ \tableau{ \\ 3} &
      \tableau{& 4 \\ 1 & 3     }\ \raisebox{-0.5\cellsize}{$\otimes$}\ \tableau{ \\ 2} &
      \tableau{& 4 \\ 2 & 3     }\ \raisebox{-0.5\cellsize}{$\otimes$}\ \tableau{ \\ 1} \\[2\cellsize]
      \tableau{& 3 \\ 1 & \st{2}}\ \raisebox{-0.5\cellsize}{$\otimes$}\ \tableau{ \\ 4} &
      \tableau{& 4 \\ 1 & \st{2}}\ \raisebox{-0.5\cellsize}{$\otimes$}\ \tableau{ \\ 3} &
      \tableau{& 4 \\ 1 & \st{3}}\ \raisebox{-0.5\cellsize}{$\otimes$}\ \tableau{ \\ 2} &
      \tableau{& 4 \\ 2 & \st{3}}\ \raisebox{-0.5\cellsize}{$\otimes$}\ \tableau{ \\ 1} 
      \end{array}
    \end{displaymath}
    \caption{\label{fig:tensor}The signed standard tableaux of shape $(2,1)\otimes(1)$.} 
  \end{center}
\end{figure}

We extend the \newword{descent set} to concatenated tableaux $S\otimes T\in \SST(\gamma \otimes \delta)$ via the concatenation of hook reading words $\hook(S)\hook(T)$. 

Using Stanley's fundamental theorem for $P$-partitions \cite{Sta72}, Gessel \cite{Ges84} proved that the structure constants for the fundamental quasisymmetric functions are given by the \newword{shuffle product}. 

\begin{proposition}[\cite{Ges84}]
  Given $A \subseteq[a-1]$ and $B\subseteq[b-1]$ and $\alpha,\beta$ any two words of lengths $a,b$, respectively, in \emph{disjoint} letters such that $\Des(\alpha)=A$ and $\Des(\beta)=B$, we have
  \begin{equation}
    F_{A} F_{B} = \sum_{\omega \in \alpha \shuffle \beta} F_{\Des(\omega)} ,
    \label{e:shuffle}
  \end{equation}
  where $\alpha\shuffle\beta$ is the set of words of length $a+b$ such that both $\alpha$ and $\beta$ appear as subwords.
  \label{prop:shuffle}
\end{proposition}

Proposition~\ref{prop:shuffle} allows us to compute the fundamental expansion of a product of Schur $P$-functions as follows.

\begin{corollary}
  For $\gamma,\delta$ strict partitions, we have
    \begin{equation}
      P_{\gamma} (X) P_{\delta} (X) \ = \ \sum_{S\otimes T \in \SST(\gamma \otimes \delta)} F_{\Des(S)}(X) .
      \label{e:PPF}
    \end{equation}
    \label{cor:PPF}
\end{corollary}

For example, from Fig.~\ref{fig:tensor}, we may compute
\begin{eqnarray*}
P_{(2,1)} P_{(1)} & = & F_{\{1\}} + 2 F_{\{2\}} + F_{\{3\}} + F_{\{1,2\}} + 2 F_{\{1,3\}} + F_{\{2,3\}} .
\end{eqnarray*}

We now have the correct set-up to apply the dual equivalence machinery.

\begin{definition}
  Let $S\otimes T$ be a signed standard tableau of concatenated shape. For $1<i<n$, define the dual equivalence involutions $\psi_i$ by the rule:
  \begin{enumerate}
  \item if $i$ lies between $i-1$ and $i+1$ in $\hook(S\otimes T)$, then $\psi_i(S\otimes T) = S\otimes T$;
  \item else if $i-1, i, i+1\in S$, then $\psi_i(S\otimes T) = d_i(S)\otimes T$;
  \item else if $i-1, i, i+1\in T$, then $\psi_i(S\otimes T) = S\otimes d_i(T)$;
  \item else swap the absolute values of letters furthest apart in the hook reading word, maintaining the signs in the given cells.
  \end{enumerate}
  \label{def:deg_tensor}
\end{definition}

\begin{figure}[ht]
  \begin{center}
    \begin{tikzpicture}[xscale=2.15,yscale=1.5,
        label/.style={%
          postaction={ decorate,
            decoration={ markings, mark=at position 0.5 with \node #1;}}}]
      \node at (0,2)  (A1) {$\tableau{& 3 \\ 1 & \st{2}}\ \raisebox{-0.5\cellsize}{$\otimes$}\ \tableau{ \\ 4}$};
      \node at (1,2)  (A2) {$\tableau{& 3 \\ 1 & 2     }\ \raisebox{-0.5\cellsize}{$\otimes$}\ \tableau{ \\ 4}$};
      \node at (2,2)  (A3) {$\tableau{& 4 \\ 1 & 2     }\ \raisebox{-0.5\cellsize}{$\otimes$}\ \tableau{ \\ 3}$};
      \node at (1.5,1)(B1) {$\tableau{& 4 \\ 1 & \st{2}}\ \raisebox{-0.5\cellsize}{$\otimes$}\ \tableau{ \\ 3}$};
      \node at (0.5,1)(B2) {$\tableau{& 4 \\ 1 & \st{3}}\ \raisebox{-0.5\cellsize}{$\otimes$}\ \tableau{ \\ 2}$};
      \node at (0,0)  (C1) {$\tableau{& 4 \\ 2 & \st{3}}\ \raisebox{-0.5\cellsize}{$\otimes$}\ \tableau{ \\ 1}$};
      \node at (1,0)  (C2) {$\tableau{& 4 \\ 2 & 3     }\ \raisebox{-0.5\cellsize}{$\otimes$}\ \tableau{ \\ 1}$};
      \node at (2,0)  (C3) {$\tableau{& 4 \\ 1 & 3     }\ \raisebox{-0.5\cellsize}{$\otimes$}\ \tableau{ \\ 2}$};
      \draw[thick,color=red   ,label={[above]{$\psi_{2}$}}](A1.05) -- (A2.175) ;
      \draw[thick,color=blue,label={[above]{$\psi_{3}$}}]  (A2)    -- (A3) ;
      \draw[thick,color=red   ,label={[above]{$\psi_{2}$}}](B2.05) -- (B1.175) ;
      \draw[thick,color=blue,label={[below]{$\psi_{3}$}}]  (B2.355)-- (B1.185) ;
      \draw[thick,color=blue,label={[above]{$\psi_{3}$}}]  (C1)    -- (C2) ;
      \draw[thick,color=red   ,label={[above]{$\psi_{2}$}}](C2.05) -- (C3.175) ;
      \draw[thick,color=violet,label={[below]{$\psi_{0}$}}](A1.355)-- (A2.185) ;
      \draw[thick,color=violet,label={[right]{$\psi_{0}$}}](A3)    -- (B1) ;
      \draw[thick,color=violet,label={[left ]{$\psi_{0}$}}](B2)    -- (C1) ;
      \draw[thick,color=violet,label={[below]{$\psi_{0}$}}](C2.355)-- (C3.185) ;
      \node at (3,2)  (a1) {$\tableau{\\ 2}\ \raisebox{-0.5\cellsize}{$\otimes$}\ \tableau{& 4 \\ 1 & \st{3}}$};
      \node at (4,2)  (a2) {$\tableau{\\ 1}\ \raisebox{-0.5\cellsize}{$\otimes$}\ \tableau{& 4 \\ 2 & \st{3}}$};
      \node at (5,2)  (a3) {$\tableau{\\ 1}\ \raisebox{-0.5\cellsize}{$\otimes$}\ \tableau{& 4 \\ 2 & 3}$};
      \node at (4.5,1)(b1) {$\tableau{\\ 2}\ \raisebox{-0.5\cellsize}{$\otimes$}\ \tableau{& 4 \\ 1 & 3}$};
      \node at (3.5,1)(b2) {$\tableau{\\ 3}\ \raisebox{-0.5\cellsize}{$\otimes$}\ \tableau{& 4 \\ 1 & 2}$};
      \node at (3,0)  (c1) {$\tableau{\\ 3}\ \raisebox{-0.5\cellsize}{$\otimes$}\ \tableau{& 4 \\ 1 & \st{2}}$};
      \node at (4,0)  (c2) {$\tableau{\\ 4}\ \raisebox{-0.5\cellsize}{$\otimes$}\ \tableau{& 3 \\ 1 & \st{2}}$};
      \node at (5,0)  (c3) {$\tableau{\\ 4}\ \raisebox{-0.5\cellsize}{$\otimes$}\ \tableau{& 3 \\ 1 & 2}$};
      \draw[thick,color=red   ,label={[above]{$\psi_{2}$}}](a1.05) -- (a2.175) ;
      \draw[thick,color=blue,label={[above]{$\psi_{3}$}}]  (a2)    -- (a3) ;
      \draw[thick,color=red   ,label={[above]{$\psi_{2}$}}](b2.05) -- (b1.175) ;
      \draw[thick,color=blue,label={[below]{$\psi_{3}$}}]  (b2.355)-- (b1.185) ;
      \draw[thick,color=blue,label={[above]{$\psi_{3}$}}]  (c1)    -- (c2) ;
      \draw[thick,color=red   ,label={[above]{$\psi_{2}$}}](c2.05) -- (c3.175) ;
      \draw[thick,color=violet,label={[below]{$\psi_{0}$}}](a1.355)-- (a2.185) ;
      \draw[thick,color=violet,label={[right]{$\psi_{0}$}}](a3)    -- (b1) ;
      \draw[thick,color=violet,label={[left ]{$\psi_{0}$}}](b2)    -- (c1) ;
      \draw[thick,color=violet,label={[below]{$\psi_{0}$}}](c2.355)-- (c3.185) ;
    \end{tikzpicture}
    \caption{\label{fig:P-product}Queer dual equivalence for $\SST((2,1)\otimes(1))$ and $\SST((1)\otimes(2,1))$.}
  \end{center}
\end{figure}
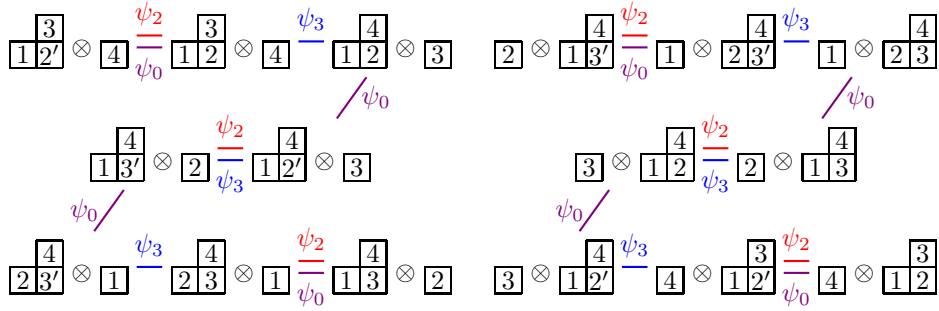

\begin{theorem}
  For $\gamma,\delta$ strict partitions of total size $n$, the maps $\psi_i$, for $1<i<n$, give a well-defined dual equivalence for $\SST(\gamma\otimes\delta)$.
  \label{thm:deg-tensor}
\end{theorem}

\begin{proof}
  To see that $\psi_i$ is well-defined, note that while neither $S$ nor $T$ is a signed standard tableau, Definition~\ref{def:deg_shifted} can be applied whenever all three entries with absolute values $i-1,i,i+1$ lie in the same tableau. Condition (i) of Definition~\ref{def:deg} is most easily verified by computer, since one needs only check a finite number of cases based on the relative diagonals of up to six consecutive letters. Condition (ii) of Definition~\ref{def:deg} is immediate from the fact that $\psi_i$ considers and affects only the positions of entries with absolute values $i-1,i,i+1$, and $\{i-1,i,i+1\} \cap \{j-1,j,j+1\} = \varnothing$ whenever $|i-j| \geq 3$.   
\end{proof}

Theorem~\ref{thm:deg-tensor} establishes the Schur positivity of a product of Schur $P$-functions. To get Schur $P$-positivity, we need only add a single queer involution.

\begin{definition}
  Let $S\otimes T$ be a signed standard tableau of concatenated shape. Define the queer dual equivalence involution $\psi_0$ by the rule:
  \begin{enumerate}
  \item if $1,2\in S$, then $\psi_0(S\otimes T) = d_0(S)\otimes T$;
  \item else if $1,2\in T$, then $\psi_0(S\otimes T) = S\otimes d_0(T)$;
  \item else swap $1$ and $2$.
  \end{enumerate}
  \label{def:qdeg_tensor}
\end{definition}

For examples of the involutions on concatenated shapes, see Figs.~\ref{fig:P-product} and \ref{fig:P-product2}. Comparing this with Fig.~\ref{fig:qdeg-31}, we see that indeed we have $P_{(2,1)} P_{(1)} = P_{(3,1)}$. Note that while we obviously have $P_{(2,1)} P_{(1)} = P_{(1)} P_{(2,1)}$, the queer dual equivalence structures for $\SST((2,1)\otimes(1))$ and $\SST((1)\otimes(2,1))$ are not in obvious bijection.

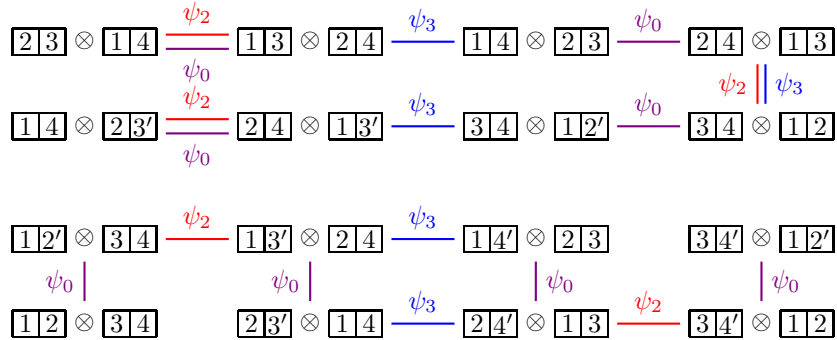
\begin{figure}[ht]
  \begin{center}
    \begin{tikzpicture}[xscale=3,yscale=1.5,
        label/.style={%
          postaction={ decorate,
            decoration={ markings, mark=at position 0.5 with \node #1;}}}]
      \node at (0,2.75) (A1) {$\tableau{2 & 3}\ \raisebox{0.3\cellsize}{$\otimes$}\ \tableau{1 & 4}$};
      \node at (1,2.75) (A2) {$\tableau{1 & 3}\ \raisebox{0.3\cellsize}{$\otimes$}\ \tableau{2 & 4}$};
      \node at (2,2.75) (A3) {$\tableau{1 & 4}\ \raisebox{0.3\cellsize}{$\otimes$}\ \tableau{2 & 3}$};
      \node at (3,2.75) (B1) {$\tableau{2 & 4}\ \raisebox{0.3\cellsize}{$\otimes$}\ \tableau{1 & 3}$};
      \node at (3,2)    (B2) {$\tableau{3 & 4}\ \raisebox{0.3\cellsize}{$\otimes$}\ \tableau{1 & 2}$};
      \node at (2,2)    (C1) {$\tableau{3 & 4}\ \raisebox{0.3\cellsize}{$\otimes$}\ \tableau{1 & \st{2}}$};
      \node at (1,2)    (C2) {$\tableau{2 & 4}\ \raisebox{0.3\cellsize}{$\otimes$}\ \tableau{1 & \st{3}}$};
      \node at (0,2)    (C3) {$\tableau{1 & 4}\ \raisebox{0.3\cellsize}{$\otimes$}\ \tableau{2 & \st{3}}$};
      \draw[thick,color=red   ,label={[above]{$\psi_{2}$}}](A1.05) -- (A2.175) ;
      \draw[thick,color=blue,label={[above]{$\psi_{3}$}}]  (A2)    -- (A3) ;
      \draw[thick,color=red   ,label={[left]{$\psi_{2}$}}](B2.100) -- (B1.260) ;
      \draw[thick,color=blue,label={[right]{$\psi_{3}$}}]  (B2.80)-- (B1.280) ;
      \draw[thick,color=blue,label={[above]{$\psi_{3}$}}]  (C1)    -- (C2) ;
      \draw[thick,color=red   ,label={[above]{$\psi_{2}$}}](C3.05) -- (C2.175) ;
      \draw[thick,color=violet,label={[below]{$\psi_{0}$}}](A1.355)-- (A2.185) ;
      \draw[thick,color=violet,label={[above]{$\psi_{0}$}}](A3)    -- (B1) ;
      \draw[thick,color=violet,label={[above]{$\psi_{0}$}}](B2)    -- (C1) ;
      \draw[thick,color=violet,label={[below]{$\psi_{0}$}}](C3.355)-- (C2.185) ;
      \node at (0,0.25) (T1) {$\tableau{1 & 2}\ \raisebox{0.3\cellsize}{$\otimes$}\ \tableau{3 & 4}$};
      \node at (0,1)    (T2) {$\tableau{1 & \st{2}}\ \raisebox{0.3\cellsize}{$\otimes$}\ \tableau{3 & 4}$};
      \node at (1,1)    (T3) {$\tableau{1 & \st{3}}\ \raisebox{0.3\cellsize}{$\otimes$}\ \tableau{2 & 4}$};
      \node at (2,1)    (T4) {$\tableau{1 & \st{4}}\ \raisebox{0.3\cellsize}{$\otimes$}\ \tableau{2 & 3}$};
      \node at (1,0.25) (U3) {$\tableau{2 & \st{3}}\ \raisebox{0.3\cellsize}{$\otimes$}\ \tableau{1 & 4}$};
      \node at (2,0.25) (U4) {$\tableau{2 & \st{4}}\ \raisebox{0.3\cellsize}{$\otimes$}\ \tableau{1 & 3}$};
      \node at (3,0.25) (U5) {$\tableau{3 & \st{4}}\ \raisebox{0.3\cellsize}{$\otimes$}\ \tableau{1 & 2}$};
      \node at (3,1)    (U6) {$\tableau{3 & \st{4}}\ \raisebox{0.3\cellsize}{$\otimes$}\ \tableau{1 & \st{2}}$};
      \draw[thick,color=red   ,label={[above]{$\psi_{2}$}}](T2) -- (T3) ;
      \draw[thick,color=blue,label={[above]{$\psi_{3}$}}]  (T3) -- (T4) ;
      \draw[thick,color=blue,label={[above]{$\psi_{3}$}}]  (U3) -- (U4) ;
      \draw[thick,color=red   ,label={[above]{$\psi_{2}$}}](U4) -- (U5) ;
      \draw[thick,color=violet,label={[left]{$\psi_{0}$}}](T1) -- (T2) ;
      \draw[thick,color=violet,label={[left]{$\psi_{0}$}}] (T3) -- (U3) ;
      \draw[thick,color=violet,label={[right]{$\psi_{0}$}}](T4) -- (U4) ;
      \draw[thick,color=violet,label={[right]{$\psi_{0}$}}](U5) -- (U6) ;
    \end{tikzpicture}
    \caption{\label{fig:P-product2}Queer dual equivalence for $\SST((2)\otimes(2))$.}
  \end{center}
\end{figure}

\begin{theorem}
  For $\gamma,\delta$ strict partitions of total size $n$, the maps $\psi_i$, for $i=0,2,\ldots,n-1$, give a well-defined queer dual equivalence for $\SST(\gamma\otimes\delta)$. 
  \label{thm:qdeg-tensor}
\end{theorem}

\begin{proof}
  To see that $\psi_0$ is well-defined, note that while neither $S$ nor $T$ is a signed standard tableau, Definition~\ref{def:d_0} applies whenever $1$ and $2$ or $\st{2}$, whichever exists, are both in the same tableau. When this is not the case, both must be the leftmost entry of the first row, in which case neither can be primed. Therefore $\psi_0$ is well-defined. By Theorem~\ref{thm:deg-tensor}, it remains to check conditions (i)-(iii) of Definition~\ref{def:queer-deg}.

  Condition (i) is verified by visual inspection of Figs.~\ref{fig:P-product} and \ref{fig:P-product2} together with the graphs on $\SST((3)\otimes(1))$ and on $\SST((1)\otimes(3))$.

  Condition (ii) of Definition~\ref{def:queer-deg} holds since $\psi_0$ considers and affects only entries $1$ and $2$, and $\psi_i$ for $i>1$ considers and affects only the positions of entries with absolute values $i-1,i,i+1$, and $\{1,2\} \cap \{i-1,i,i+1\} = \varnothing$ whenever $i\geq 4$.

  Condition (iii) we prove by induction on $n$. For $n \leq 4$, we appeal to condition (i) or directly to the graphs in Figs.~\ref{fig:P-product} and \ref{fig:P-product2} and on $\SST((3)\otimes(1))$ and $\SST((1)\otimes(3))$. Assuming $n\geq 5$, we must show that we need at most two instances of $\psi_{n-1}$ to connect any two elements of $\SST(\gamma\otimes\delta)$. By the analogous condition for dual equivalence, we use $\psi_{n-1}$ at most once within a dual equivalence class, and we apply it one additional time, if needed, to move $n$ between the two concatenated diagrams.
\end{proof}

In particular, by Theorem~\ref{thm:queer-deg}, Corollary~\ref{cor:PPF}, and Theorem~\ref{thm:qdeg-tensor}, we have an elementary proof of the Schur $P$-positivity of products of Schur $P$-functions.


%
\section{Odd involutions}
%
\label{sec:crystals}

The main impediment to further applications of queer dual equivalence is condition (iii) of Definition~\ref{def:queer-deg}. For dual equivalence, \cite{Ass15}(Theorem~4.2) proves the analogous condition for dual equivalence graphs is equivalent to the \emph{local} condition (i) of Definition~\ref{def:deg}. Since the queer involution commutes with all but the first two dual equivalence involutions, no local characterization involving just the dual equivalence and queer dual equivalence involutions can exist.

In order to overcome this issue, given a queer dual equivalence $\{\psi_0\} \cup \{\psi_i\}_{1<i<n}$, one can attempt to define \emph{odd dual equivalence involutions} $\{\psi^{\prime}_i\}_{1<i<n}$ as specific compositions of the dual equivalence involutions with the queer involution such that the following \emph{local} conditions imply \emph{global} Schur $P$-positivity:
  \renewcommand{\theenumi}{\roman{enumi}}
  \begin{enumerate}
  \item For all $0 \leq i-h \leq 3$ and all $S \in \mathcal{A}$, there exists a strict partition $\gamma$ of $i-h+3$ such that
    \[ \sum_{U \in [S]_{(h,i)}} F_{\Des_{(h-1,i+1)}(U)}(X) = P_{\gamma}(X), \]
    where $[S]_{(h,i)}$ is the equivalence class generated by $\psi_h,\psi^{\prime}_h,\ldots,\psi_i,\psi^{\prime}_i$.

  \item For all $|i-j| \geq 3$ and all $S \in\mathcal{A}$, we have
    \begin{eqnarray} 
      \psi^{\prime}_{i} \psi^{\prime}_{j}(S) & = & \psi^{\prime}_{j} \psi^{\prime}_{i}(S), \\
      \psi_{i} \psi^{\prime}_{j}(S) & = & \psi^{\prime}_{j} \psi_{i}(S) .
    \end{eqnarray}
    
  \end{enumerate}

To begin, we may set $\psi^{\prime}_2 = \psi_0$, and 
\begin{equation}
  \psi^{\prime}_3(S) = \left\{ \begin{array}{rl}
    \psi_0 \psi_2 (S) & \mathrm{if} \psi_2 \psi_0(S) = \psi_0(S), \\
    \psi_2 \psi_0 (S) & \mathrm{if} \psi_0 \psi_2(S) = \psi_0(S), \\
    S & \mathrm{if} \psi_2(S) = \psi_0(S)      .
  \end{array}\right.
\end{equation}
Beyond this, however, the definition of $\psi^{\prime}_4$ is not uniquely forced, nor does there appear to be so simple a definition that works for $\psi^{\prime}_5$. Nevertheless, one can (and the author has) manually construct a working solution up to $n=7$, which suggests that a general rule might well be found.

Further evidence for this approach comes from realizing Schur $P$-polynomials as characters of tensor representations of the queer Lie superalgebra. Lie superalgebras are algebras with a $\mathbb{Z}/2\mathbb{Z}$ grading, allowing for two families of variables (one commuting and one not) to interact. Originally arising from mathematical physics in connection with supersymmetry, Lie superalgebras were formalized mathematically and classified by Kac \cite{Kac77}. One well-studied superalgebra generalization of the general linear Lie algebra is the queer superalgebra. Quantized universal enveloping algebras were developed for the queer superalgebra by Sergeev \cite{Ser84}, with the corresponding crystal theory developed by Grantcharov, Jung, Kang, Kashiwara, and Kim \cite{GJKK10,GJKKK10}. The latter also gave an explicit construction of the queer crystal on semistandard decomposition tableaux \cite{Ser10}, an alternative combinatorial model for Schur $P$-polynomials developed by Serrano. Assaf and Oguz \cite{AO20} give a direct construction of the queer crystal on semistandard shifted tableaux.

The queer crystal is described in terms of the usual type A raising operators $\{e_i\}_{1 \leq i < n}$ along with a queer crystal operator $e_0$. By considering certain crystal automorphisms $S_i$ which respect the braid relations, one defines the \emph{odd crystal operators} $\{e_{i'}\}_{1 \leq i < n}$ by $e_{1'}=e_0$ and
\begin{equation} 
  e_{i'} = S_{w_i^{-1}} e_{0} S_{w_i} 
\end{equation}
for $i>1$, where $w_i = s_2 \cdots s_i s_1 \cdots s_{i-1}$ is the shortest (in coxeter length) permutation such that $w_i \cdot \alpha_i = \alpha_1$ (here $\alpha_i$ is the vector with $1$ in position $i$ and $-1$ in position $i+1$). These odd crystal operators appear to satisfy similar local axioms to the type A crystal operators as characterized by Stembridge \cite{Ste03}. Assaf and Oguz \cite{AO20} gave a partial characterization for these operators to determine a queer crystal. Gillespie, Hawkes, Poh, and Schilling \cite{GHPS20} added additional conditions that are sufficient but lack the local nature of Stembridge's axioms.

Assaf proved a direct relationship between type A crystals and dual equivalence graphs \cite{Ass08} that shows the equivalence between Stembridge's local characterization and the local axioms for dual equivalence. Thus one can expect that local axioms for the queer crystal might inform local conditions for dual equivalence, and conversely. At present, both problems remains open.

%
%

\bibliographystyle{amsalpha} 
\bibliography{queerDEG}

\end{document}